\documentclass[14pt,a4paper]{amsart}
\usepackage[utf8]{inputenc}
\usepackage[english]{babel}
\usepackage{graphicx}
\usepackage{xparse}
\usepackage{amsmath}
\usepackage{esint}
\usepackage{amssymb}
\usepackage{amsfonts}
\usepackage{mathrsfs}
\usepackage{fancyhdr}
\usepackage{indentfirst}
\usepackage{newlfont}
\usepackage{latexsym}
\usepackage{amsthm}
\usepackage{caption}
\usepackage{subfig}
\usepackage{listings}
\usepackage{xcolor}
\usepackage{appendix}
\usepackage{chngpage}
\usepackage{calc}
\usepackage{booktabs}
\usepackage[italicdiff]{physics}
\usepackage{fullpage}

\theoremstyle{plain} 
\newtheorem{thm}{Theorem}[section] 

\newtheorem{lem}[thm]{Lemma} 
\newtheorem{prop}[thm]{Proposition}      
\newtheorem{rmk}[thm]{Remark}

\theoremstyle{definition}             
\newtheorem{Def}{Definition}[section] 
\newtheorem{ese}{Example}[section]    

\newcommand{\s}{\sigma}
\newcommand{\Set}[2]{\mathcal{S}_{\Omega}(B_{#1},{#2}f)}
\newcommand{\SetCentro}[3]{\mathcal{S}_{\Omega}(B_{#1}({#2}),{#3}f)}
\newcommand{\e}{\varepsilon}
\newcommand{\R}{\mathbb{R}}
\newcommand{\Kf}{\mathbb{K}_{\Omega,f}}
\newcommand{\K}{\mathbb{K}}
\newcommand{\ol}[1]{\overline{#1}}

\newcommand{\diam}[1]{\text{diam}({#1})}

\begin{document}

\title[Abstract approach to non homogeneous Harnack inequality]{Abstract approach to non homogeneous Harnack inequality in doubling quasi metric spaces}
\author[Chiara Guidi\& Annamaria Montanari ]{Chiara Guidi \& Annamaria Montanari}
\address{Dipartimento di Matematica, Universit\`a di Bologna,
piazza di Porta S.Donato 5, 40127 Bologna, Italy.}
\email{chiara.guidi12@unibo.it,
annamaria.montanari@unibo.it}

\vspace{5mm}

\begin{abstract}
We develop an abstract theory to obtain Harnack inequality for non homogeneous PDEs in the setting of quasi metric spaces. The main idea is to adapt the notion of double ball and critical density property given by Di Fazio,  Gutiérrez, Lanconelli, taking into account the right hand side of the equation. Then we apply the abstract procedure to the case of subelliptic equations in non divergence form involving Grushin vector fields and to the case of X-elliptic operators in divergence form.

\vspace{5mm}

 %35J70  Degenerate elliptic equations
 %35B45  A priori estimates
 %35B65 Smoothness and regularity of solutions. 

\vspace{1mm}
   Non divergence subelliptic PDEs with measurable coefficients. Subelliptic equations. Grushin vector fields. Carnot–Carathéodory metric.
\end{abstract}

\maketitle

\section{Introduction}
Since its introduction in 1887 for positive solutions of harmonic functions, Harnack inequality has been extended to a huge variety of PDEs both in divergence and non divergence form and in Euclidean and non Euclidean setting. The importance of this type of inequality is widely recognized as it is a fundamental tool in the investigation of regularity of solutions of partial differential equations. Moser iteration technique (\cite{JM}) for elliptic operators in divergence form and Krylov-Safonov's measure theoretic approach (\cite{KS79, KS80}) are cornerstones in the development of Harnack inequality procedures. The works \cite{C89} and \cite{CC95} by Caffarelli, where the Krylov-Safonov's technique is simplified, and its extension to the linearized Monge-Ampère equation in \cite{CG97}, enlightened the quasi-metric character of the proof and inspired axiomatic procedures to Harnack inequality in the general setting of quasi metric spaces (see \cite{AFT, DGL, IMS} ). The key idea of these axiomatic procedures is to define a particular functional set and find sufficient conditions for the set to imply Harnack inequality for each member of the family of functions. This kind of approach was used by Aimar, Forzani and Toledano in \cite{AFT} to prove Harnack inequality for some class of continuous functions while Di Fazio, Gutiérrez and Lanconelli in \cite{DGL} substituted the regularity assumption with some more natural requirement on the geometry of the quasi metric space (ring condition, Definition \ref{ring condition}). Despite different assumptions, both results rely on covering lemmas and states that the Harnack inequality is a consequence of the power like decay property of the distribution function of each member of the considered family. While  Aimar, Forzani, Toledano and Di Fazio, Gutiérrez, Lanconelli proved the power-like decay property as a consequence of double ball and critical density property (see \cite{DGL} for the definitions), Indratno,  Maldonado and Silwal in \cite{IMS} substituted the first property with an integral condition (see \cite[(2.14)]{IMS} ) which makes their approach better suited for variational operators. The authors need to require the doubling quasi-metric  space $(X,d,\mu)$ (Definitions \ref{def Holder cont} and \ref{doubling}) to be such that continuous functions are dense in $L^1(X, d\mu)$.
%These three abstract procedure have been widely exploited to prove Harnack type inequalities for different operators. 
The usefulness of these three abstract procedures is attested by the variety of works that exploited these approaches to prove Harnack type inequalities, see for example \cite{AGT, AM, T, GT}.
Unfortunately none of the theoretical approaches to Harnach inequality obtained in \cite{AFT, DGL, MS} can be directly applied when handling a PDE with non zero right hand side, indeed they can only handle functional set closed under multiplication by positive constant, while the set of solutions of a non homogeneous PDE does not satisfy that requirement. However, a remark is in order. 
If a maximum principle  holds true in a form that
permits to establish pointwise-to-measure estimates for super solutions of the non homogeneous equation
and  if the existence of a solution for the Dirichlet problem for the corresponding homogeneous PDE is guaranteed,
then it is possible to obtain the non homogeneous Harnack inequality from the homogeneous one by  an elementary argument as in \cite[Theorem 5.5]{GL}.
However, when dealing  with subelliptic PDEs in non divergence form, it is not known if a maximum principle holds true in a form that
permits to establish pointwise-to-measure estimates for super solutions such as
\cite[Theorem 2.1.1]{G16}.
For non divergence PDEs its proof depends
in a crucial way upon the maximum principle of Aleksandrov, Bakelman
and Pucci, see for example \cite[Section 9.8]{GT01}. This principle for some subelliptic
PDEs have been recently  studied,  for example in \cite{CG97, C97,  AM, DGN03, GM04, BCK15, TZ}.  However, 
to the best of our knowledge, nowadays this is still an interesting open problem for subelliptic PDEs in non divergence form and
this lack precludes one from extending the method of \cite[Theorem 5.5]{GL}
 to obtain a non homogeneous Harnack inequality in  general subelliptic settings. 

The purpose of this work is to obtain an axiomatic approach to Harnack inequality that permits to handle both homogeneous and non homogeneous equations in divergence and non divergence form in the general setting of quasi metric spaces. To this aim we modify the notion of double ball and critical density properties given in \cite{DGL},  taking into account the right hand side of the equation. 
Our main results are Theorem \ref{teo: power decay} and Theorem \ref{Harnack}, where we prove an invariant Harnack inequality for non homogeneous sets of functions 
satisfying the double ball (see Definition \ref{def: DBP}) and the critical density property (see Definition \ref{def crit dens}).
Our technique is an extension to non homogeneous sets of functions  of the abstract procedure introduced  in \cite[Section 4 and 5]{DGL}. 
This extension has been inspired by \cite[Lemma 4.5]{CC95}, where the critical density property is proved for viscosity solutions of non homogeneous uniformly elliptic equations with small enough $L^n$ norm of the right hand side.\\

The article is organized as follows. In Section 2 we recall some definitions and useful results of quasi metric set theory, then we introduce the notions of double ball and critical density property and, with the aid of an $\varepsilon$-Besicovich type lemma presented in \cite{DGL}, we will prove that these two properties imply power decay under two different sets of hypotheses (Theorem \ref{teo: power decay} ). Moreover we show that Harnack inequality is a consequence of the power decay property (Theorem \ref{Harnack}). In the last two sections we apply the abstract theory to the case of subelliptic equations in non divergence form with measurable  coefficients involving Grushin vector fields and to the case of $X-$elliptic operators in divergence form. The first application is a new result extending the one obtained by the second author in \cite{AM}. We explicitly remark that the non homogeneous double ball property is more complicated than the homogeneous one and it requires the construction of ad hoc barriers and the use of the weighted ABP property (see Theorem \ref{thm ABP}).
The second application is a new proof of the result obtained by Uguzzoni in \cite{U}.
\section{Abstract Harnack inequality}\label{sect abstract H}
The main result of our work is contained in this section, where we develop an abstract theory to obtain Harnack inequality for non homogeneous PDEs in the setting of quasi metric spaces. The idea is to define particular families of functions (usually non negative solutions of a PDE) 
and to prove that if those families satisfy two properties (namely the double ball and the critical density) plus some other conditions on the quasi metric space, then solutions of the considered PDE satisfy Harnack inequality. We remark that our result is a generalization  of the one obtained by Di Fazio, Gutiérrez and Lanconelli in \cite{DGL} where an abstract theory for homogeneous Harnack inequality is established.

\subsection{Definitions and Preliminaries}

We recall some definitions and set the notation that will be extensively used in the sequel.

\begin{Def}(Quasi distance)\label{quasi distance}
Let $Y\neq \emptyset$, we say that a function $d:Y\times Y\to [0,+\infty[$ is a quasi distance if
\begin{itemize}
\item for every $x,y\in Y$, $d(x,y)=d(y,x)$
\item for every $x,y\in Y$, $d(x,y)=0$ if and only if $x=y$
\item (quasi triangular inequality) there exists a constant $K\geq 1$ such that $$d(x,y)\leq K(d(x,z)+d(z,y))\quad \text{ for every } x,y,z \in Y.$$
\end{itemize}
\end{Def}
\noindent In this case the pair $(Y,d)$ is called quasi metric space and the set 
\begin{equation*}
B_r(x):=\{y\in Y: d(x,y)<r \}
\end{equation*}
\noindent is called a $d$-ball of center $x\in Y$ and radius $r>0$ (by abuse of notation we will write ``ball" instead of ``$d$-ball"). It is well known that a quasi metric $d$ on a set $Y$ induces a topology on $Y$ such that the balls $B_r(y)$ form a basis of neighborhoods.
%and that balls may not be open sets. 

\begin{Def} \cite[p.66]{CW}
We say that the quasi metric space $(Y,d)$ is of homogeneous type if the balls $B_r(y)$ are a basis of open neighborhoods and there exists a positive integer $N$ such that for each $x$ and $r>0$, the ball $B_r(x)$ contains at most $N$ points $x_j$ such that $d(x_i,x_j)\geq r/2$ with $i\neq j$.
\end{Def}

\begin{Def}\label{def Holder cont}(H\"older quasi-distance)
We say that the quasi distance $d$ is H\"older continuous if there exist positive constants $\beta$ and $0<\alpha\leq 1$ such that
\begin{equation}\label{diseq: Holder cont}
|d(x,y)-d(x,z)|\leq \beta d(y,z)^{\alpha}(d(x,y)+d(x,z))^{1-\alpha}\quad\text{for all }x,y,z\in Y.
\end{equation}
In this case, the pair $(Y,d)$ is said to be a H\"older quasi metric space.
\end{Def} 
In the sequel we will always assume $d$ to be a H\"older quasi-distance. This requirement on the regularity of $d$ is not a restrictive assumption since Marcías and Segovia \cite[Theorem 2]{MS} proved that given a quasi distance $d$ it is always possible to construct an equivalent quasi distance $d'$ which is H\"older continuous and such that $d'$-balls are open sets with respect the topology induced by $d'$. %This results make non restrictive our assumption about the regularity of $d$.

We now introduce the notion of doubling property on a quasi-metric space $(Y,d)$ with a positive measure $\mu$ defined on a $\sigma$-algebra $\mathscr{A}$ on $Y$ containing $d$-balls.
%Consider a quasi metric space $(Y,d)$ and a $\sigma$-algebra $\mathscr{A}$ on $Y$ containing $d$-balls. Then, given a positive measure $\mu$ on $\mathscr{A}$ we introduce the notion of doubling quasi metric space that will be crucial in the sequel.  
\begin{Def}\label{doubling}(Doubling)\label{def doubling space}
We say that the measure $\mu$ satisfies the doubling property if there exists a constant $C_D>0$ such that
\begin{equation*}
0<\mu(B_{2r}(x))\leq C_D \mu(B_r(x)),\quad \text{ for all } x\in Y\text{ and } r>0
\end{equation*}
or equivalently
\begin{equation*}
\mu(B_{r_2}(x))\leq C_D\biggl(\frac{r_2}{r_1}\biggl)^q\mu(B_{r_1}(x)), \quad \text{where } q=\log_2C_D.
\end{equation*}
In this case the triple $(Y,d,\mu)$ is called a doubling quasi metric space.
\end{Def}
%
%\begin{Def}
%We say that $(Y,d,\mu)$ is a doubling quasi metric space if $(Y,d)$ is a quasi metric space and $\mu$ has the doubling property.
%\end{Def}

We remark that any H\"older doubling quasi metric space is always separable (see  \cite{MS}), consequently open sets are measurable as they are countable union of $d$-balls.

In the sequel we will also use the following ring condition on the doubling quasi metric space $(Y,d,\mu).$

\begin{Def}(Ring condition)\label{ring condition}
We say that the doubling quasi metric space $(Y,d,\mu)$ satisfies the ring condition if there exists a non negative function $\omega(\e)$ such that $\omega(\e)\to 0$ as $\e\to 0^+$ and for every ball $B_r(x)$ and all $\e>0$ sufficiently small we have
\begin{equation*}
\mu(B_r(x)\setminus B_{(1-\e)r}(x))\leq\omega(\e)\mu(B_r(x)).
\end{equation*}
\end{Def}
\subsection{Double Ball Property, Critical Density and Power Decay}
Consider $(Y,d,\mu)$ a H\"older doubling quasi metric space and $\Omega\subseteq Y$ open. In order to build a non homogeneous abstract Harnack inequality we consider families of functions which depend on $\Omega$ and another variable that takes into account the non homogeneous part of the PDE.

\begin{Def}
Let $P(\Omega)$ be the power set of $\Omega$ equipped with a partial order $\preceq$. Given $S_1$ and $S_2\in P(\Omega)$ we say that the set function $\mathcal{S}:P(\Omega)\to \mathbb{R}$ is order preserving if $S_1\preceq S_2$ implies $\mathcal{S}(S_1)\leq \mathcal{S}(S_2)$.
\end{Def}

Let $\mathcal{B}=\{B_r(x): x\in\Omega,\; r>0\;\text{and}\: B_r(x)\subset\Omega\}$ be the set of all the quasi metric balls contained in $\Omega$ ordered by inclusion, and 
$\mathcal{F}(\Omega)=\{ f:\Omega \rightarrow \mathbb{R} 
 \text{ such that } f \text{ is }\text{measurable} \}$.

%=\{\lambda f: \lambda \in\mathbb{R} \}$. 
Hereafter we consider  a function

$$\mathcal{S}_{\Omega}:\mathcal{B}\times \mathcal{F}(\Omega)\to [0,+\infty[$$

such that
\begin{itemize}
\item $\mathcal{S}_{\Omega}$ is order preserving with respect to the first variable, i.e. $\mathcal{S}_{\Omega}(B_R(x), f)\geq \mathcal{S}_{\Omega}(B_r(y), f)$ for every $B_r(y)\subseteq B_R(x)$ and for every $f\in \mathcal{F}(\Omega)$.

\item For every $\lambda\in\mathbb{R},$ $f\in \mathcal{F}(\Omega)$ and $B_r(x)\subseteq \Omega$, we have $\mathcal{S}_{\Omega}(B_r(x),\lambda f)=|\lambda|\mathcal{S}_{\Omega}(B_r(x),f)$. 
\end{itemize}
\begin{Def}
We say that $f\in \mathcal{L}(\Omega)$ if $f\in \mathcal{F}(\Omega)$ and 
$
\mathcal{S}_{\Omega}(B_r(x),f)<+\infty$ for every $B_r(x)\subseteq \Omega.$
\end{Def}

\begin{ese}
For example,  for the function $\mathcal{S}_{\Omega}(B_r(x),f)=\left(\int_{B_r(x)}|f|^p\right)^{1/p}$ we have 
that $\mathcal{L}(\Omega)=L^p_{loc}(\Omega).$

\end{ese}

\begin{Def}\label{def KOmega}
For $f\in \mathcal{L}(\Omega)$ we define $\mathbb{K}_{\Omega, f}$ a family of non negative measurable functions  
\begin{equation*}
\mathbb{K}_{\Omega, f}\subset \{u:A\to\mathbb{R} \text{ such that } A\subset\Omega, u\geq 0 \text{ and } u \text{ is }\text{measurable} \}
\end{equation*}
such that the following two conditions hold:
\begin{itemize}
\item If $u\in \mathbb{K}_{\Omega, f}$ then $\lambda u\in \mathbb{K}_{\Omega, \lambda f}$ for all  $\lambda\geq 0.$
\item If $u\in \mathbb{K}_{\Omega, f}$ then for every $\lambda,\tau\geq 0$ such that  $\tau- \lambda u\geq 0$ we have $\tau- \lambda u \in \mathbb{K}_{\Omega, -\lambda f}$.

\end{itemize}

\end{Def}
In particular if $u\in\Kf$ and its domain contains $A\subset \Omega$ we write $u\in \Kf(A)$. The next two sections, where a Harnack inequality for a specific operator $L$ is obtained, will clarify the role of the families of functions $\mathbb{K}_{\Omega, f}$. Roughly speaking, $\mathbb{K}_{\Omega,f}$ will contain all nonnegative the measurable solutions $u$ with domain contained in $\Omega$ of an equation of the type  $Lu=f,$ where 
$L$ is a second order partial differential operator.

\begin{Def}(Structural constant)\label{def structural constant} We say that $c$ is a structural constant if it is independent of each $u$ belonging to the family $\Kf$, of $f\in\mathcal{L}(\Omega)$ and of the balls defined by the quasi distance considered.
\end{Def}

In what follows we do not specify the center of a ball if the center is the point $x_0$, namely we write $B_R$ instead of $B_R(x_0)$.
 
\begin{Def}(Critical density)\label{def crit dens}
Let $0<\nu<1$. We say that $\Kf$ satisfies the $\nu$ critical density property if there exist structural constants  $0<\e_{CD}, c<1$ depending on $\nu$ and $\eta_{CD}>1$ such that for every ball $B_{\eta_{CD}R}\subset\Omega$ and for every $u\in\Kf(B_{\eta_{CD}R})$ with
\begin{equation*}
\mu(\{x\in B_R:u(x)\geq 1\})\geq \nu\mu(B_R),
\end{equation*} 
we have
\begin{equation*}
\inf_{B_{R/2}}u\geq c\quad\text{or}\quad \Set{\eta_{CD}R}{}\geq\e_{CD}.
\end{equation*}
In this case we say that  $\Kf$ satisfies the $\nu$ critical density property $CD (\nu,c,\e_{CD},\eta_{CD})$.
\end{Def}

\begin{rmk}\label{rmk eps crit density}
If $\Kf$ satisfies the $\ol{\nu}$ critical density property  $CD (\ol{\nu},c,\e_{CD},\eta_{CD})$, then $\Kf$ satisfies $CD (\nu,c,\e_{CD}, \eta_{CD})$ for any $\nu>\ol{\nu}.$
\end{rmk}

\begin{Def}(Double ball property) \label{def: DBP}
We say that $\Kf$ satisfies the double ball property if there exist structural constants $0<\e_{DB}, \gamma<1$ and $\eta_{DB}>1$ such that for every $B_{\eta_{DB}R}\subset\Omega$ and for every $u\in\Kf(B_{\eta_{DB}R})$ with 
\begin{equation*}
\inf_{B_{R/2}}u\geq 1\quad\text{and}\quad \Set{\eta_{DB}R}{}< \e_{DB}
\end{equation*}
we have  
\begin{equation*}
\inf_{B_R}u\geq \gamma.
\end{equation*}
In this case we say that  $\Kf$ satisfies the double ball property $DB (\gamma,\e_{DB},\eta_{DB})$.
\end{Def}

It is not restrictive to require the critical density and the double ball property to hold with the same constants $\eta_{CD}=\eta_{DB}$ and $\e_{CD}=\e_{DB}$, indeed we have the following remark

\begin{rmk}\label{rmk cd db stesso eps }
If $\Kf$ satisfies the $\nu$ critical density property  $CD (\nu,c,\e_{CD},\eta_{CD})$ and the double ball property $DB (\gamma,\e_{DB},\eta_{DB})$ then it satisfies $CD (\nu,c,\e,\eta)$ and $DB (\gamma,\e,\eta)$ with $\e=\min\{\e_{DB},\e_{CD}\}$ and $\eta=\max\{\eta_{DB},\eta_{CD}\}$.
\end{rmk}

\begin{Def}(Power decay)\label{def:power decay}
We say that the family of functions $\Kf$ satisfies the power decay  property if there exist structural constants $0\leq \gamma,\e_{P}<1$ and  $\eta_{P},\;M>1$  such that for each $u\in\Kf(B_{\eta_{P} R})$ with 
\begin{equation*}
\inf_{B_R}u\leq 1\quad\text{and}\quad  \Set{\eta_{P}R}{}<\e_{P}
\end{equation*}
we have
\begin{equation*}
\mu(\{x\in B_{R/2}:u(x)>M^k\})\leq \gamma^k\mu(B_{R/2})\quad\text{for every } k\in\mathbb{N} .
\end{equation*}
In this case we say that $\Kf$ satisfies the power decay property $PD (M,\gamma,\e_{P},\eta_{P})$.
\end{Def}

We aim to prove that in a H\"older doubling quasi metric space if $\Kf$ satisfies the double ball and critical density property plus some other conditions then it also satisfies the  power decay property. To show this we need some preliminary results. 

%We remark that in the proofs of the following Propositions we use the constants appearing in their statement without recalling every time their definition. 

\begin{prop} Let $C_D$ be the doubling constant, if $\Kf$ satisfies the $\nu$ critical density property $CD(\nu,c,\e_{CD},\eta_{CD})$ for some $0<\nu<1/C_D^2$, then $\Kf$ satisfies the double ball property $DB(c,\e_{CD},2\eta_{DB})$.
\end{prop}
\begin{proof}
Suppose by contradiction that there exists $u\in\Kf (B_{2\eta_{CD}R})$, $B_{2\eta_{CD}R}\subset\Omega$ , such that $\inf_{B_{R/2}}u\geq 1$ and $\Set{2\eta_{CD}R}{}<\e_{CD}$ but $\inf_{B_R}u<c$. Then by the $\nu$ critical density property we have 
\begin{equation*}
\mu(\{x\in B_{2R}:u(x)\geq 1\})< \nu\mu(B_{2R})\quad \text{or}\quad  \Set{2\eta_{CD}R}{}\geq\e_{CD}.
\end{equation*} 
If the second inequality holds we have an immediate contradiction. Otherwise if the first inequality holds, since $B_{R/2}\subseteq \{x\in B_{2R}:u(x)\geq 1\}$, we find
\begin{equation*}
\mu (B_{R/2})\leq \mu(\{x\in B_{2R}:u(x)\geq 1\})\leq \nu\mu(B_{2R})\leq \nu C_D^2\mu(B_{R/2})<\mu(B_{R/2})
\end{equation*}
a contradiction.
\end{proof}

\begin{prop}\label{prop: critical dens alpha}
Suppose $\K_{\Omega, f}$ satisfies the double ball  and the $\nu$ critical density properties  $DB (\gamma,\e,\eta)$ and $CD (\nu,c,\e,\eta)$ for every $f \in \mathcal{L}(\Omega).$
Then there exists a structural constant $M_0=\frac{1}{\gamma c}>1$ such that for any positive constant $\alpha$ and for any $u\in\Kf(B_{\eta R})$ with
\begin{equation*}
\mu(\{x\in B_R:u(x)\geq \alpha\})\geq \nu\mu( B_R),
\end{equation*} 
we have \begin{equation*}%\label{tesi prop critical dens alpha}
\inf_{B_{R}}u\geq \frac{\alpha}{M_0}\quad\text{or}\quad \Set{\eta R}{}\geq \e\alpha c.
\end{equation*}
\end{prop}

\begin{proof}
Since $u\in\Kf(B_{\eta R})$ we have $\frac{u}{\alpha}\in\K_{\Omega,\frac{f}{\alpha}}(B_{\eta R})$. By the $\nu$ critical density property of $\K_{\Omega,\frac{f}{\alpha}}$ follows either $\inf_{B_{R/2}}\frac{u}{\alpha}\geq c$ or $\Set{\eta R}{}\geq\alpha \e\geq \alpha c \e$. In the second case we are done, otherwise $\frac{u}{\alpha c}\in \K_{\Omega,\frac{f}{\alpha c}}(B_{\eta R})$ and so, either $\Set{\eta R}{}\geq \alpha c \e $ or we can apply the double ball property to obtain $\inf_{B_R}\frac{u}{\alpha c}\geq \gamma$. Defining $\gamma c:=\frac{1}{M_0}$ we conclude the proof.
\end{proof}

%\begin{rmk}\label{rmk norma f small ball}
%If there exists a positive structural constant $C$ such that 
%\begin{equation*}
%\Set{R}{}<C
%\end{equation*}
%then
%\begin{equation*}
%\SetCentro{r}{x}{}<C\quad\text{for every }B_{r}(x)\subset B_R.
%\end{equation*}
%\end{rmk}

\begin{lem}\label{lemma raggio rho}
Suppose the same hypotheses of Proposition \ref{prop: critical dens alpha} hold, $u\in \Kf(B_{\theta R})$,
$$\inf_{B_R}u\leq 1$$
and there exist constants $\alpha>0$, $\rho<2KR$ and $y\in B_R$ such that 
\begin{equation}\label{hyp lemma raggio rho}
\mu(\{x\in B_{\rho}(y):u(x)\geq\alpha\})\geq\nu\mu(B_{\rho}(y))
\end{equation}
where $ \theta:=K(1+4\eta K)>1$ and $K$ is the constant in the quasi triangle inequality.\\
Then there exist positive structural constants $\s,\;M_1$ such that 
\begin{equation}\label{tesi lemma raggio rho}
\rho\leq \biggl(\frac{M_1}{\alpha}\biggr)^{\s}R\quad\text{or}\quad  \Set{\theta R}{}\geq \frac{\e\alpha c \gamma^p}{M_0}.
\end{equation}
Here $M_1=(4K)^{1/\s}M_0$, $M_0=\frac{1}{\gamma c}$ is defined in Proposition \ref{prop: critical dens alpha}, $\s=-\frac{\log2}{\log\gamma}$ and $p\in\mathbb{N}$ is chosen so that $2^{p-1}\rho\leq 2KR\leq 2^p\rho$.
\end{lem}

\begin{proof}
If the second inequality in \eqref{tesi lemma raggio rho} holds, the proof is completed. Thus we suppose $\Set{\theta R}{}<\frac{\e\alpha c \gamma^p}{M_0}$ which implies
\begin{equation}\label{normaf lemma raggio rho }
\SetCentro{r}{\tilde{y}}{} <\frac{\e\alpha c \gamma^p}{M_0} \quad \text{for every } B_r(\tilde{y})\subseteq B_{\theta R}.
\end{equation}
Since $B_{\eta\rho}(y)\subset B_{\theta R}$ and inequality \eqref{hyp lemma raggio rho} holds, we apply Proposition \ref{prop: critical dens alpha} and taking into account \eqref{normaf lemma raggio rho } we deduce
\begin{equation}\label{inf1 lemma raggio rho}
\inf_{B_{\rho}(y)}u\geq \frac{\alpha}{M_0}.
\end{equation}
Moreover if $p$ is chosen as in the statement, since $y\in B_R$, it follows $B_{2^{p}\eta\rho}(y)\subseteq B_{\theta R}$ and so \eqref{normaf lemma raggio rho } implies 
\begin{equation}\label{normaf2 lemma raggio rho}
\SetCentro{2^{k+1}\eta\rho}{y}{}<\frac{\e\alpha c\gamma^p}{M_0} \leq \frac{\e\alpha c\gamma^k}{M_0}\quad\text{for every }0\leq k\leq p-1,
\end{equation}
hence we can repeatedly apply the double ball property to $\frac{u M_0}{\alpha\gamma^k}$ in $B_{2^{k+1}\eta \rho}(y)$, where $k=0,\dots,p-1$, obtaining
\begin{equation}\label{infp lemma raggio rho}
\inf_{B_{2^p\rho}(y)}u\geq \gamma^p \frac{\alpha}{M_0}.
\end{equation} 
Indeed by \eqref{inf1 lemma raggio rho} we have $\inf_{B_{\rho}(y)}\frac{uM_0}{\alpha}\geq 1$, this together with \eqref{normaf2 lemma raggio rho} allow us to use the double ball property of $K_{\Omega,f \frac{M_0}{\alpha} }$ to get $\inf_{B_{2\rho}(y)}u\frac{M_0}{\alpha}\geq \gamma$. Now we have $\inf_{B_{2\rho}(y)}u\frac{M_0}{\alpha\gamma}\geq 1$. Again, by \eqref{normaf2 lemma raggio rho} and the double ball property, $\inf_{B_{4\rho}(y)}u\frac{M_0}{\alpha}\geq \gamma^2$. We repeat this procedure $p$ times to find \eqref{infp lemma raggio rho} and consequently
\begin{equation*}
1\geq \inf_{B_R}u\geq \inf_{B_{2^p\rho}(y)}u\geq \gamma^p\frac{\alpha}{M_0}.
\end{equation*}
From the first and the last inequality in the expression above we get $\gamma^p\leq \frac{M_0}{\alpha}$. Since $\gamma^{\s}=1/2$, raising both side of the inequality to the power $\s$, it follows $2^{-p}\leq \bigl(\frac{M_0}{\alpha}\bigr)^{\sigma}$. Finally, multiplying both sides by $2^p\rho$ and keeping in mind the definition of $p$ in the statement we get the thesis.
\end{proof}

\begin{lem}\label{lemma raggio rho sharp}
Under the same hypotheses of Lemma \ref{lemma raggio rho} we have
\begin{equation*}\label{tesi lemma raggio rho sharp}
\rho\leq \biggl(\frac{M_1}{\alpha}\biggr)^{\s}R\quad\text{or}\quad \Set{\theta R}{}\geq \e c.
\end{equation*}
\end{lem}

\begin{proof}
We shall prove that if  $\Set{\theta R}{}<\e c$ then $\delta:= \frac{\alpha \gamma^p}{M_0}\leq 1.$
Indeed, if $\Set{\theta R}{}< \e c \delta $ then by the proof of the previous lemma $\delta\leq 1.$ 
On the other side, if $\e c \delta \leq \Set{\theta R}{}< \e c,$ then obviously $\delta <1.$
Since in both cases $\delta\leq 1,$ then $\rho\leq \bigl(\frac{M_1}{\alpha}\bigr)^{\s}R.$
\end{proof}
For reader convenience we recall here the definition of a density point that will be used in the proof of the following power decay property.
\begin{Def}(Density point) We say that $x\in Y$ is a density point for $X\subset Y$ if $\frac{\mu(B_R(x)\cap X)}{\mu(B_r(x))}\to 1$ as $r\to 0^+$
\end{Def}

\begin{thm}[Power decay]\label{teo: power decay}
Let $(Y,d,\mu)$ be a H\"older doubling quasi-metric space and consider $\Omega\subset Y$ open, $f:\Omega\to \mathbb{R}$. Suppose there exists $0< \delta <1$ such that $\mu(B_r(x))\leq \delta\mu(B_{2r}(x))$ for every $B_{2r}(x)\subset\Omega$ and one of the following pairs of conditions holds
\begin{itemize}
\item[(A)]
\begin{itemize}
\item[(A1)]$\K_{\Omega, f}$ satisfies the double ball and the $\nu$ critical density properties $DB (\gamma,\e,\eta)$ and $CD (\nu,c,\e,\eta)$ for every $f\in \mathcal{L}(\Omega)$.

\item[(A2)] $(Y,d,\mu)$ satisfies the ring conditions with $\omega(s)=o(\log^{-2}(1/s))$ as $s\to 0^+$.
\end{itemize}
or
\item[(B)]
\begin{itemize}
\item[(B1)]$\K_{\Omega, f}$ satisfies the $\nu$ critical density properties $CD (\nu,c,\e,\eta)$ for $0<\nu<1/C_D^2$ and for every $f\in \mathcal{L}(\Omega)$. Here $C_D$ is the doubling constant.
\item[(B2)] The function $r\to\mu(B_r(x))$ is continuous.
\end{itemize}
\end{itemize} 
Then the family $\Kf$ satisfies the power decay property.
\end{thm}

\begin{proof}
First of all we define $\tau:=\max\{\nu,\delta\}$, by Remarks \ref{rmk eps crit density}, $\K_{\Omega,\lambda f}$ satisfies the critical density property $CD (\tau,c, \nu,\eta)$ . That said, throughout the proof we will write $\nu$ instead of $\tau$.\\
In order to prove the theorem under assumptions (A1), (A2) we consider $u\in\K_{\Omega,f}(B_{\eta_P r})$ and set 
\begin{equation*}
E_k=\{x\in B_{\eta_P r}:u(x)\geq M^k\}, \quad \text{for every } k\in \mathbb{N}
\end{equation*} 
and suppose 
\begin{equation}\label{eq:hyp crit dens}
B_{\eta_P r}\subset\Omega,\quad\inf_{B_r} u \leq 1\quad\text{and}\quad \Set{\eta_P r}{}<\e_P
\end{equation}
where $\eta_P, M>1$ and $0<\e_P<1$ are structural constant that will be soon determined.
We shall prove that there exists a structural constant $0<\omega<1$ such that 
\begin{equation*}
\mu(\{x\in B_{r/2}:u(x)>M^k\})\leq \omega^k\mu(B_{r/2})\quad\text{for every } k\in\mathbb{N}. 
\end{equation*}
Notice that the second inequality in \eqref{eq:hyp crit dens} implies 
\begin{equation}\label{normaf teo power decay}
\SetCentro{\rho}{x}{}<\e_P\quad \text{for every } B_{\rho}(x)\subseteq B_{\eta_P r}.
\end{equation}
We claim it is possible to construct a family of balls $B_k$ of radius $t_k$ and center $x_0$ such that $r=t_0>t_1>t_2>\dots> r/2$ and 
\begin{equation}\label{diseq : misura Ek Bk}
\mu(B_{k+1}\cap E_{k+2})\leq c(\nu)\mu(B_{k}\cap E_{k+1}),\quad c(\nu)<1,\quad k\in\mathbb{N}_0.
\end{equation}
In particular we will construct this family so that $t_k=T_kr$ where $T_k$ are defined by
\begin{equation}\label{eq def Tk}
\begin{cases}
T_k=T_1-\beta_1q^3\sum_{j=0}^{k-2}q^j,\quad k>2\\
T_2=3/4-\beta_1q^3\\
T_1=3/4
\end{cases}
\quad\text{i.e.}\quad
\begin{cases}
T_{k+1}=T_k-\beta_1q^{k+2},\quad k>1\\
T_1=3/4.\\
\end{cases}
\end{equation}
Here 
\begin{equation}\label{eq def q b1}
q:=1/M^{\s\alpha},\quad \beta_1:=(2K)^{\alpha-1}\beta M_1^{\s\alpha}(1+M_1^{\s})^{1-\alpha},
\end{equation}
$\s$, $M_1$ are defined in Lemma \ref{lemma raggio rho}, $\alpha,\;\beta$ are the constants that appearing in Definition \ref{def Holder cont} and $K$ is the constant in the quasi triangular inequality. 
Assuming the claim for a moment, from \eqref{diseq : misura Ek Bk}, we get
\begin{align*}
\mu(\{x\in B_{r/2}:u(x)>M^{k+2} \})&\leq \mu(\{x\in B_{t_{k+1}}:u(x)>M^{k+2}\})\\
&\leq (c(\nu))^{k+1}\mu(B_r)\\
&\leq (c(\nu))^{k+1}C_D\mu(B_{r/2}) \quad \text{for every } k\in\mathbb{N}.
\end{align*}
where $C_D$ is the doubling constant. Consider a positive integer $k_0$ such that $(c(\nu))^{k_0}C_D<1$, if we define $\tilde{M}=M^{k_0+2}$, from the last inequality we have
\begin{align*}
\mu(\{x\in B_{r/2}:u(x)>\tilde{M}^t \})&\leq \mu(\{x\in B_{r/2}:u(x)>M^{k_0+1+t}\})\\
&\leq (c(\nu))^{k_0}C_D(c(\nu))^t \mu(B_{r/2}) \\
&\leq (c(\nu))^t\mu(B_{r/2})\quad \text{for every } t\in\mathbb{N}.
\end{align*}
Replacing $\tilde{M},\; t$ with $M,\; k$ respectively we get the thesis.\\
We now explicitly define $\eta_P$ and $M$ (these specific choices will be motivated in the proof of the claim) such that
\begin{gather*}
\eta_P>\max\{K(3\eta K+1),\theta\}\\
M>\max\{M_0,\ol{M}\},
\end{gather*}
where $\ol{M}$ is a positive constant big enough to obtain $\max\Bigl\{\beta_1\frac{1}{\ol{M}^{2\s\alpha}}, \beta_1\frac{1}{\ol{M}^{3\s\alpha}}\sum_{j\in\mathbb{N}} \frac{1}{\ol{M}^{j\s\alpha}}\Bigr\}<\frac{1}{4}$ and $M_0$ and $\theta$ are defined in Proposition \ref{prop: critical dens alpha} and Lemma \ref{lemma raggio rho} respectively. We remark that the definitions of $M$ and $T_k$ imply
\begin{equation}\label{bound Tk}
\frac{1}{2}< T_k\leq \frac{3}{4}\quad \text{for every }k\in\mathbb{N}.
\end{equation}
To complete the proof of the Theorem \ref{teo: power decay}, we are left with the proof of claim \eqref{diseq : misura Ek Bk}. We will explicitly show it in the case $k=0$ and then for a generic $k\in\mathbb{N}$, for the sake of clarity we will subdivide each proof in five steps.\\ 
Proof of the claim for $k=0$

\begin{itemize}

\item{(STEP I)} Consider $t_1=T_1r=3/4r$. Since $B_1\cap E_2\subset B_r\subset B_{2r}\subset \Omega$ the result of Di Fazio et al. \cite[Theorem 3.3]{DGL} ensures the existence of a level $\nu$ covering $\mathcal{F}_1=\{B(x_h,r_h)\}$ of $B_1\cap E_2$ where $x_h$ are density points of $B_1\cap E_2$, $r_h<3Kr$ for every $h\in \mathbb{N}$ and 
\begin{equation}\label{diseq cover B1E2}
\nu=\frac{\mu(B_{r_h}(x_h)\cap B_1\cap E_2)}{\mu(B_{r_h}(x_h))}\leq \frac{\mu(B_{r_h}(x_h)\cap E_2)}{\mu(B_{r_h}(x_h))}.
\end{equation}

\item{(STEP II)} We show that    
\begin{equation*}
B_{r_h}(x_h)\subset E_1,\quad\text{for every } h\in\mathbb{N}.
\end{equation*}
Since $\eta_P>K(3\eta K+1)$ and $r_h<3Kr$ we have $B_{\eta r_h}(x_h)\subset B_{\eta_P r}$ so that $u\in \Kf(B_{\eta r_h}(x_h))$. By \eqref{diseq cover B1E2} and Proposition \ref{prop: critical dens alpha} it follows $\inf_{B_{r_h}(x_h)}u\geq M^2/M_0>M$ or $\SetCentro{\eta r_h}{x_h}{}\geq \e c M^2$. It suffices to choose 
$$\e_P=\e c$$ 
in \eqref{eq:hyp crit dens} and recall that, by definition, $M>M_0>1$ to exclude the latter alternative and get $B_{r_h}(x_h)\subset E_1$.

\item{(STEP III)} Now we prove $r_h<2Kr$ for all $h\in\mathbb{N}$. Suppose by contradiction that there exists a $j\in\mathbb{N}$ such that $r_j>2Kr$. Then $B_{2r}(x_j)\subset B_{2Kr}(x_j)\subset B_{r_j}(x_j)$. By Step II, $\inf_{B_{r_j}(x_j)}u\geq M^2/M_0> 1$ and so $\inf_{B_{2r}(x_j)}u> 1$. In addition, since $x_j$ is a density point of $B_1\cap E_2$, it follows that $x_j\in \ol{B_{t_1}}$. Finally, recalling that $t_1<r$, we have $B_r\subset B_{2Kr}(x_j)$ and consequently $\inf_{B_{r}}u> 1$. This contradicts \eqref{eq:hyp crit dens}.

\item{(STEP IV)} We  next show that  
\begin{equation*}
B_{r_h}(x_h)\subset B_{t_0},\quad\text{for every } h\in\mathbb{N}.
\end{equation*}
First of all notice that since $\eta_P>\theta$ we have $u\in\Kf(B_{\theta r})$. This, \eqref{diseq cover B1E2} and Step III allow us to apply Lemma \ref{lemma raggio rho sharp} with $y$, $\rho$ and $\alpha$ replaced by $x_h$, $r_h$ and $M^2$, obtaining 
\begin{equation*}
r_h\leq \biggl(\frac{M_1}{M^2}\biggr)^{\s}r \quad \text{or} \quad\Set{\theta r}{}\geq \e c.
\end{equation*}
Since $\e_P=\e c$ and \eqref{normaf teo power decay} holds, we conclude the first alternative take place. Now, if $z\in B_{r_h}(x_h)$, inequality \eqref{diseq: Holder cont} and the quasi triangular inequality imply
\begin{align*}
d(z,x_0)&\leq d(x_h,x_0)+\beta(d(x_h,z))^{\alpha}(d(x_h,x_0)+d(x_0,z))^{1-\alpha}\\
&\leq d(x_h,x_0)+(2K)^{\alpha-1}\beta(d(x_h,z))^{\alpha}(d(x_h,x_0)+d(x_h,z))^{1-\alpha}\\
&\leq t_1+(2K)^{\alpha-1}\beta\biggl(\frac{M_1}{M^2}\biggr)^{\s\alpha}r^{\alpha}\biggl(t_1+\biggl(\frac{M_1}{M^2}\biggr)^{\s}r\biggr)^{1-\alpha}.
\end{align*}
keeping in mind $t_1=T_1r$ we get
\begin{align*}
d(z,x_0)&\leq r\biggl(T_1+(2K)^{\alpha-1}\beta q^2 M_1^{\sigma\alpha}\bigl(T_1+M_1^{\sigma}q^{2/\alpha}\bigr)^{1-\alpha}\biggr)\\
&\leq r\biggl(T_1+(2K)^{\alpha-1}\beta q^2 M_1^{\sigma\alpha}\bigl(1+M_1^{\sigma}\bigr)^{1-\alpha}\biggr)\\
&\leq r(T_1+\beta_1 q^2)
\end{align*}
where $\beta_1$ e $q$ are the positive constants defined in \eqref{eq def q b1}. In virtue of our choice of $t_1$ and $M$ we have $T_1+\beta_1q^2<1$ and so, $B_{r_h}(x_h)\subset B_r$ concluding the proof of Step IV.

\item{(STEP V)}  By \cite[Theorem 3.3(iv)]{DGL} we have
\begin{equation*}
\mu(B_1\cap E_2)\leq c(\nu)\mu\biggl(\bigcup_{h\in\mathbb{N}}B_{r_h}(x_h)\biggr),
\end{equation*}
on the other hand Step II and Step VI imply $B_{r_h}(x_h)\subset B_0\cap E_1$ and hence
\begin{equation*}
c(\nu)\mu\biggl(\bigcup_{h\in\mathbb{N}}B_{r_h}(x_h)\biggr)\leq c(\nu)\mu(B_0\cap E_1), 
\end{equation*}
combining inequalities above we get \eqref{diseq : misura Ek Bk} for $k=0$.
\end{itemize}
Proof of the claim for a generic $k\in\mathbb{N}$
\begin{itemize}

\item{(STEP I)} Consider $t_{k+1}=T_{k+1}r$, where $T_k$ is defined in \eqref{eq def Tk}. Since $B_{k+1}\cap E_{k+2}\subset B_r\subset B_{2r}\subset \Omega$ the result of Di Fazio et al. \cite[Theorem 3.3]{DGL} ensures the existence of a level $\nu$ covering $\mathcal{F}_{k+1}=\{B_{r_h}(x_h)\}$ of $B_{k+1}\cap E_{k+2}$ where $x_h$ are density point of $B_{k+1}\cap E_{k+2}$, $r_h<3Kr$ for every $h\in \mathbb{N}$ and 
\begin{equation}\label{diseq cover B2E3}
\nu= \frac{\mu(B_{r_h}(x_h)\cap B_{k+1}\cap E_{k+2})}{\mu(B_{r_h}(x_h))}\leq \frac{\mu(B_{r_h}(x_h)\cap E_{k+1})}{\mu(B_{r_h}(x_h))}.
\end{equation}

\item{(STEP II)} We show that    
\begin{equation*}
B_{r_h}(x_h)\subset E_{k+1},\quad\text{for every } h\in\mathbb{N}.
\end{equation*}
Since $\eta_P>K(3\eta K+1)$ and $r_h<3Kr$ we have $B_{\eta r_h}(x_h)\subset B_{\eta_P r}$ so that $u\in \Kf(B_{\eta r_h}(x_h))$. From \eqref{diseq cover B2E3} and Proposition \ref{prop: critical dens alpha} it follows that $\inf_{B_{r_h}(x_h)}u\geq M^{k+2}/M_0$ or $\SetCentro{\eta r_h}{x_h}{}\geq \e c M^{k+2}$. By \eqref{normaf teo power decay}, the definition of $\e_P$, and our choice of $M$ we exclude the latter alternative and get $B_{r_h}(x_h)\subset E_{k+1}$.

\item{(STEP III)} Now we prove that $r_h<2Kr$ for all $h$. Suppose by contradiction that there exists a $j\in\mathbb{N}$ such that $r_j>2Kr$. Then $B_{2r}(x_j)\subset B_{2Kr}(x_j)\subset B_{r_j}(x_j)$. By step II $\inf_{B_{r_j}(x_j)}u\geq M^{k+2}/M_0> 1$ so that $\inf_{B_{2r}(x_j)}u> 1$. In addition, since $x_j$ is a density point of $B_{k+1}\cap E_{k+2}$, it follows that $x_j\in \ol{B_{t_{k+1}}}$. Finally, recalling that $t_{k+1}<r$, we have $B_r\subset B_{2Kr}(x_j)$ and consequently $\inf_{B_{r}}u> 1$, on the other hand  \eqref{eq:hyp crit dens} holds and we reach a contradiction.

\item{(STEP IV)} We  next show that  
\begin{equation*}
B_{r_h}(x_h)\subset B_k,\quad\text{for every } h\in\mathbb{N}.
\end{equation*}
First of all we notice that since $\eta_P>\theta$ we have $u\in\Kf(B_{\theta r})$. This, the second inequality in \eqref{diseq cover B2E3} and Step III allow us to apply Lemma \ref{lemma raggio rho sharp} with $y$, $\rho$ and $\alpha$ replaced by $x_h$, $r_h$ and $M^{k+2}$ obtaining 
\begin{equation*}
r_h\leq \biggl(\frac{M_1}{M^{k+2}}\biggr)^{\s}r \quad \text{or} \quad \Set{\theta r}{}\geq \e c.
\end{equation*}
By \eqref{normaf teo power decay} and the choice of $\e_P$ we made, we can conclude that the first alternative take place. Now, if $z\in B_{r_h}(x_h)$, inequality \eqref{diseq: Holder cont} and the quasi triangular inequality imply
\begin{align*}
d(z,x_0)&\leq d(x_h,x_0)+\beta(d(x_h,z))^{\alpha}(d(x_h,x_0)+d(x_0,z))^{1-\alpha}\\
&\leq d(x_h,x_0)+(2K)^{1-\alpha}\beta(d(x_h,z))^{\alpha}(d(x_h,x_0)+d(x_h,z))^{1-\alpha}\\
&\leq t_{k+1}+(2K)^{1-\alpha}\beta\biggl(\frac{M_1}{M^{k+2}}\biggr)^{\s\alpha}r^{\alpha}\biggl(t_{k+1}+\biggl(\frac{M_1}{M^{k+2}}\biggr)^{\s}r\biggr)^{1-\alpha}.
\end{align*}
Moreover, since $t_{k+1}=T_{k+1}r$ we have
\begin{align*}
d(z,x_0)&\leq r\biggl(T_{k+1}+(2K)^{1-\alpha}\beta q^{k+2}M_1^{\sigma\alpha}\bigl(T_{k+1}+ M_1^{\sigma}q^{(k+2)/\alpha}\bigr)^{1-\alpha}\biggr)\\
&\leq r\biggl(T_{k+1}+(2K)^{1-\alpha}\beta q^{k+2}M_1^{\sigma\alpha}\bigl(1+ M_1^{\sigma}\bigr)^{1-\alpha}\biggr)\\
&\leq r(T_{k+1}+\beta_1 q^{k+2})
\end{align*}
where $\beta_1$ e $q$ are the positive constant defined in \eqref{eq def q b1}. Keeping in mind \eqref{eq def Tk} and \eqref{bound Tk}, we have $T_{k+1}+\beta_1q^{k+2}=T_k<1$ and so $B_{r_h}(x_h)\subset B_r$ concluding the proof of Step IV.

\item{(STEP V)} One one hand by of \cite[Theorem 3.3(iv)]{DGL} we have
\begin{equation*}
\mu(B_{k+1}\cap E_{k+2})\leq c(\nu)\mu\biggl(\bigcup_{h\in\mathbb{N}}B_{r_h}(x_h)\biggr),
\end{equation*}
on the other hand Step II and Step VI imply $B_{r_h}(x_h)\subset B_{k}\cap E_{k+1}$ and hence
\begin{equation*}
c(\nu)\mu\biggl(\bigcup_{k\in\mathbb{N}}B_{r_h}(x_h)\biggr)\leq c(\nu)\mu(B_k\cap E_{k+1}), 
\end{equation*}
combining inequalities above we get \eqref{diseq : misura Ek Bk} for a generic $k\in\mathbb{N}$.
\end{itemize}

This proves the claim and completes the proof Theorem \ref{teo: power decay} under assumptions (A1) and (A2).\\
Now suppose hypotheses (B) holds, the proof proceeds exactly as before with $\eta_P:=2\max\{K(3 \eta K+1),\theta\}$; by using \cite[Theorem 3.4]{DGL} instead of \cite[Theorem 3.3]{DGL}. Moreover \eqref{diseq cover B1E2} and \eqref{diseq cover B2E3} have to be replaced by
\begin{equation*}
\frac{\nu}{C_D}\leq \frac{\mu(B_{r_h}(x_h)\cap B_1\cap E_2)}{\mu(B_{r_h}(x_h))}\leq \frac{\mu(B_{r_h}(x_h)\cap E_2)}{\mu(B_{r_h}(x_h))}
\end{equation*}
and
\begin{equation*}
\frac{\nu}{C_D}\leq \frac{\mu(B_{r_h}(x_h)\cap B_{k+1}\cap E_{k+2})}{\mu(B_{r_h}(x_h))}\leq \frac{\mu(B_{r_h}(x_h)\cap E_{k+1})}{\mu(B_{r_h}(x_h))}
\end{equation*}
respectively.
\end{proof}

%\begin{prop} Suppose the hypotheses of Theorem \ref{teo: power decay} hold true and $\K_{\Omega, f}$ satisfies the double ball and the $\nu$ critical density properties $DB (\gamma,\e,\eta)$ and $CD (\nu,c,\e,\eta)$ for every $f\in \mathcal{L}(\Omega)$ with the same constants $\gamma$, $\nu$, $c$, $\e$ and $\eta$ independently of $f$. (Or alternatively $\K_{\Omega, f}$ satisfies the $\nu$ critical density properties $CD (\nu,c,\e,\eta)$ for $0<\nu<1/C_D^2$ and for every $f\in \mathcal{L}(\Omega)$ with the same constants $\nu$, $c$, $\e$ and $\eta$ independently of $f$).
%Then the family $\K_{\Omega, f}$ satisfies the power decay property $PD(M,\omega,\eta_P,\e_P)$ with  $M,\omega,\e_P,\eta_P$ independent of $f$.
%\end{prop}
%\begin{proof}
%Let $u\in\K_{\Omega,f}$. For every $\tau\geq 0$ and $\lambda\in\mathbb{R}$ such that $\tau+\lambda u\geq 0$ it suffices to rename $u_{\lambda}:=\tau+\lambda u$, $f_{\lambda}:= \lambda f$, and apply Theorem \ref{teo: power decay} to $u_{\lambda}\in \K_{\Omega,f_{\lambda}}$ to prove the power decay property for $\K_{\Omega,\lambda f}$. We point out that since we require $\K_{\Omega,\lambda f}$ to have both the double ball and the $\nu$-critical density property for every $\lambda \in\mathbb{R}$ it is straightforward to apply Theorem above.
%\end{proof}

\subsection{Proof of Abstract Harnack Inequality}
Our goal is to prove the following Harnack inequality
\begin{thm}[Harnack inequality]\label{Harnack}
Let $(Y,d,\mu)$ be a doubling quasi metric space, suppose $\K_{\Omega, f}$ satisfies the power decay property $PD (M,\gamma,\e_P,\eta_P) $  for every $f\in \mathcal{L}(\Omega)$.\\
Then, for every $B_{\eta R}(x_0)\subset\Omega $, if  $u\in \Kf(B_{\eta R}(x_0))$ is non negative and locally bounded, there exists a positive structural constant $C$ such that
\begin{equation*}%\label{tesi Harnack}
\sup_{B_R(x_0)} u\leq C\biggl(\inf_{B_R(x_0)}u+\SetCentro{\eta R}{x_0}{}\biggr),
\end{equation*}
where $\eta=2K(2K\eta_P+1)$ and $K$ is the constant in the quasi triangle inequality.
\end{thm}

We remark that in virtue to \cite[Theorem 2]{MS} we avoid requiring $d$ to be a H\"older quasi distance since there always exists an equivalent quasi distance $d'$ that satisfies this property. We prove Harnack inequality using the following Lemma and Proposition.

\begin{lem}\label{lemma Harnack}
Under the same hypotheses of Theorem \ref{Harnack}, if $\SetCentro{2\eta_P R}{z_0}{}\leq \e_P$, $u\in\Kf(B_{2\eta_P R}(z_0))$ is such that $\inf_{B_{2R}(z_0)}u\leq 1$, $u(x_0)\geq M^k$ and $B_{2\rho}(x_0)\subset B_R(z_0)$, then
\begin{equation}\label{tesi lemma Harnack}
\sup_{B_{\rho}(x_0)}u\geq u(x_0)\biggl(1+\frac{1}{M}\biggr).
\end{equation}
Here $x_0\in B_R(z_0)$, $k\geq 2$, $\rho=\frac{\gamma^{k/q}R}{c_1}$, $q=\log_2(C_D)$, $c_1<\biggl(\frac{\gamma^{1/q}(1-\gamma)^{1/q}}{C_D^{1/q}4K\eta_P}\biggr)$, $C_D$ is the doubling constant and $K$ appears in the quasi triangle inequality.
\end{lem}

\begin{proof}
We shall prove the statement by contradiction. Suppose \eqref{tesi lemma Harnack} is not true and define 
\begin{gather*}
A_1:=\{x\in B_R(z_0):u(x)\geq M^{k-1} \},\\
A_2:=\{x\in B_{\rho/(2\eta_P)}(x_0):w(x)\geq M \}\\
\intertext{where}
w(x):=\frac{u(x_0)\bigl(1+\frac{1}{M}\bigr)-u(x)}{\frac{u(x_0)}{M}}=M+1-\frac{M}{u(x_0)}u(x)\in \K_{\Omega,-\frac{Mf}{u(x_0)}}(B_{\rho}(x_0)).
\end{gather*}
Since $w(x_0)=1$, $\inf_{B_{\rho/\eta_P}(x_0)}w(x)\leq 1$. Moreover $\SetCentro{\rho}{x_0}{-\frac{M}{u(x_0)}}= \SetCentro{\rho}{x_0}{}\frac{M}{u(x_0)}\leq \e_P M^{-k+1}$, so the power decay property for $\K_{\Omega,f} $ and $\K_{\Omega,\lambda f}$ implies respectively
\begin{equation*}
\mu(A_2)\leq \gamma \mu(B_{\rho/(2\eta_P)}(x_0))\quad\text{and}\quad\mu(A_1)\leq \gamma^{k-1} \mu(B_R (z_0)).
\end{equation*}
Recalling that $B_{\rho/(2\eta_P)}(x_0)\subset B_{\rho}(x_0)\subset B_R(z_0)$ we can show the inclusion $B_{\rho/(2\eta_P)}\subset A_1\cup A_2$. Indeed if $x\in B_{\rho/(2\eta_P)}$ but $x\notin A_1$ then $u(x)<M^{k-1}$ so $w(x)\geq M$ and hence $x\in A_2$, vice versa if $x\in B_r(z_0)$ but $x\notin A_2$ then $w(x)<M$ so $u(x)>u(x_0)\geq M^k$ and hence $x\in A_1$. Consequently we estimate the measure of $B_{\rho/(2\eta_P)}(x_0)$ by
\begin{equation*}
\mu(B_{\rho/(2\eta_P)}(x_0))\leq \mu(A_1)+\mu(A_2)\leq \gamma^{k-1}\mu(B_R(z_0))+\gamma\mu(B_{\rho/(2\eta_P)}(x_0)).
\end{equation*}
Now, since $B_{\rho/(2\eta_P)}(x_0)\subset B_R(z_0)\subset B_{2KR(x_0)}$, recalling definition \ref{doubling} the last inequality becomes
\begin{equation*}
\mu(B_{\rho/(2\eta_P)}(x_0))\leq \biggl(\gamma^{k-1}C_D\biggl(\frac{4KR\eta_P}{\rho}\biggr)^q +\gamma\biggr)\mu(B_{\rho/2\eta_P}(x_0)).
\end{equation*}
where $q=\log_2 C_D$. From the strict positiveness of the measure of $B_{\rho/(2\eta_P)}(x_0)$ and the definition of $\rho$ and $c_1$ given in the statement, we get
\begin{equation*}
1-\gamma\leq \gamma^{k-1}C_D\biggl(\frac{4KR\eta_P}{\rho}\biggr)^q=C_D(4K\eta_P c_1)^q\frac{\gamma^{k-1}}{\gamma^k}=\frac{C_D}{\gamma}(4K\eta_P c_1)^q
\end{equation*}
that is equivalent to $c_0:=\frac{C_D}{\gamma(1-\gamma)}(4K\eta_P c_1)^q\geq 1$. On the other hand, since $c_1<\Bigl(\frac{\gamma^{1/q}(1-\gamma)^{1/q}}{C_D^{1/q}4K\eta_P}\Bigr)$ we have $c_0<1$ reaching a contradiction.
\end{proof}

\begin{prop}\label{Harnack u,f small}
Suppose the same hypotheses of Theorem \ref{Harnack} are satisfied and assume $\inf_{B_R(x_0)}u<1$ and $\SetCentro{\eta R}{x_0}{}<\e_P$, then there exists a positive structural constant $C$ such that
\begin{equation*}%\label{tesi Harnack u,f small}
\sup_{B_R(x_0)} u\leq C.
\end{equation*}
\end{prop}

\begin{proof}
Consider $B_R(z)$ with $z\in B_R(x_0)$ and define
\begin{gather*}
D:=\sup_{x\in B_R(z)} u(x)g(x,R)\\
\intertext{where}
g(x,R):=\biggl(\frac{R-d(x,z)}{R}\biggr)^{\delta/\alpha},
\end{gather*}
$\delta$ is a structural constant that will be soon defined and $\alpha$ is the exponent in the H\"older property for $d$ (see Definition \ref{def Holder cont}).
We claim that $D$ is bounded from above by a structural constant C. Deferring the proof of the claim for a moment we have
\begin{equation}\label{tesi2 Harnack u,f small}
u(x)\leq C\biggl(\frac{R}{R-d(x,z)}\biggr)^{\delta/\alpha}, \quad\text{for all } x\in B_R(z) \text{ and for every } z\in B_R(x_0),
\end{equation}
thus the thesis follows taking $x=z$ in \eqref{tesi2 Harnack u,f small}.\\
Hence we are left with the proof of the claim, to that aim choose
\begin{gather*}
\delta>0\quad\text{such that }\frac{1}{M}=\gamma^{\delta/q},\\
\beta_*>2(2K)^{1-\alpha}\beta\Bigl(1-\bigl(1+\frac{1}{M}\bigr)^{-\alpha/\delta}\Bigr)^{-1}>2\beta(2K)^{1-\alpha},\\
k_0\in\mathbb{N},\quad k_0>\frac{q}{\log \gamma}\log (c_1(2^{\frac{1}{1-\alpha}}-1)),\\
\end{gather*} 
where $M,\;\rho,\; \gamma$ and $c_1$ are defined in the statement of Lemma \ref{lemma Harnack}, $\beta,\; \alpha$ is as in Definition \ref{def Holder cont}, and $q=\log_2 C_D$. Notice that $k_0$ is a structural constant whose definition implies $\bigl(1+\frac{\rho}{R}\bigr)^{1-\alpha}=\bigl(1+\frac{\gamma^{k/q}}{c_1}\bigr)^{1-\alpha}<2$ for every $k\geq k_0$ and since $u$ is non negative and locally bounded, $+\infty>D\geq 0$. If $D>0$ pick $D^*\in(0,D)$, it suffices to show that $D^*$ is bounded from above by a structural constant $C$ to prove the claim. Since $u\geq 0$ is not identically null, there exist $x_*\in B_R(z)$ such that $D^*<u(x_*)g(x_*,R)$, if $u(x_*)<1$ we are done otherwise choose $k\in \mathbb{Z}$ such that $M^k\leq u(x_*)<M^{k+1}$. For clarity sake we consider three different cases.\\
CASE 1) If $k\leq k_0$ then
\begin{equation*}
D^*<M^{k+1}g(x_*,R)\leq M^{k_0+1}.
\end{equation*}\\
CASE 2) If $k>k_0$ and $\frac{D^*}{M} c_1^{\delta}<\beta_*^{\delta/\alpha}$, clearly $D^*$ is bounded above by $\frac{M\beta_*^{\delta/\alpha}}{c_1^{\delta}}$.\\
CASE 3) $k>k_0$ and $\frac{D^*}{M} c_1^{\delta}\geq\beta_*^{\delta/\alpha}$. We will show that this is never the case.\\
For $\rho=\frac{\gamma^{k/q}R}{c_1}$ as in Lemma \ref{tesi lemma Harnack}, we have
\begin{equation*}
1\geq g(x_*,R)>\frac{D^*}{M^{k+1}}=\frac{D^*}{M}(\gamma^k)^{\delta/q}=\frac{D^*}{M}\biggl(c_1\frac{\rho}{R}\biggr)^{\delta},
\end{equation*}
hence, combining inequalities above and the definition of $g$ we compute
\begin{equation}\label{stima distanza xstar}
d(x_*,z)<R-\beta_*R^{1-\alpha}\rho^{\alpha}.
\end{equation}
Now, if $y\in B_{\rho}(x_*)$, by the H\"older property and the quasi triangular inequality, the definition of $k_0$ and $\beta_*$, for every $k\geq k_0$ we have
\begin{align*}
d(y,z)&\leq d(z,x_*)+(2K)^{1-\alpha}\beta\bigl(d(x_*,y)\bigr)^{\alpha}\bigl(d(x_*,y)+d(z,x_*)\bigr)^{1-\alpha}\\
&\leq R-\beta_*R^{1-\alpha}\rho^{\alpha}+(2K)^{1-\alpha}\beta\rho^{\alpha}(\rho+R)^{1-\alpha}\\
&\leq R-\beta_*R^{1-\alpha}\rho^{\alpha}+2(2K)^{1-\alpha}\beta R^{1-\alpha}\rho^{\alpha}\\
&<R, 
\end{align*}
hence 
\begin{equation}\label{inclusione Harnack u,f small}
B_{\rho}(x^*)\subset B_R(z).
\end{equation}
We can apply Lemma \ref{lemma Harnack} with $R$, $z_0$ and $x_0$ replaced by $KR$, $z_0$ and $x_*$ respectively to obtain 
\begin{equation}\label{app lemma Harnack}
\sup_{B_{\rho}(x_*)}u\geq u(x_*)\bigl(1+\frac{1}{M}\bigr)>\frac{D^*}{g(x_*,R)}\bigl(1+\frac{1}{M}\bigr).
\end{equation}
Indeed since $B_{2KR\eta_P}(z)\subset B_{\eta R}(x_0)$, we get $u\in \Kf(B_{2KR\eta_P}(z))$; moreover $u(x_*)\geq M^k$ and $\SetCentro{2KR\eta_P}{z}{}<\e_P$. Thus all the hypotheses of Lemma \ref{lemma Harnack} are satisfied.\\
By \eqref{inclusione Harnack u,f small}, for $y\in B_{\rho}(x_*)$
 \begin{equation}\label{dis 2 Harnack u,f small}
 \sup_{B_{\rho}(x_*)}u\leq D\sup_{y\in B_{\rho}(x_*)}\frac{1}{g(y,R)}=\frac{D}{g(x_*,R)}\sup_{y\in B_{\rho}(x_*)}\frac{g(x_*,R)}{g(y,R)}.
 \end{equation}
Moreover by the H\"older property, the quasi triangular inequality and \eqref{stima distanza xstar} we have
\begin{align*}
\Bigl(\frac{g(x_*,R)}{g(y,R)}\Bigr)^{\alpha/\delta}&=\frac{R-d(z,x_*)}{R-d(y,z)}\\
&\leq \frac{R-d(z,x_*)}{R-\bigl(d(z,x_*)+\beta(2K)^{1-\alpha}\rho^{\alpha}(d(z,x_*)+\rho)^{1-\alpha}\bigr)}\\
&\leq \frac{1}{1-\frac{\beta(2K)^{1-\alpha}\rho^{\alpha}(R+\rho)^{1-\alpha}}{\beta_*\rho^{\alpha}R^{1-\alpha}}}\\
&\leq\frac{1}{1-\frac{2\beta(2K)^{1-\alpha}}{\beta_*}}
\end{align*}
Combining \eqref{app lemma Harnack}, \eqref{dis 2 Harnack u,f small} and the inequality above we obtain
\begin{equation}\label{D* bounded Harnack u,f small}
D^*<D\Bigl(\frac{M}{1+M}\Bigr)\Bigl(\frac{\beta_*}{\beta_*-2\beta(2K)^{1-\alpha}}\Bigr)^{\delta/\alpha}.
\end{equation}
hence taking the limit for $D^*\to D$ in \eqref{D* bounded Harnack u,f small} we find
$1<\Bigl(\frac{M}{1+M}\Bigr)\Bigl(\frac{\beta_*}{\beta_*-2\beta(2K)^{1-\alpha}}\Bigr)^{\delta/\alpha}$ from which we get
\begin{align*}
\beta_*\Bigl(\Bigl(\frac{1+M}{M}\Bigr)^{\alpha/\delta}-1\Bigr)-2\beta(2K)^{1-\alpha}\Bigl(\frac{1+M}{M}\Bigr)^{\alpha/\delta}<0\\
\beta_*< 2\beta(2K)^{1-\alpha}\Bigl(1-\Bigl(\frac{M+1}{M}\Bigr)^{-\alpha/\delta}\Bigr)^{-1}
\end{align*}
which is in contrast with the previous choice of $\beta_*$.\\
\end{proof}

\begin{proof}{(Theorem \ref{Harnack})}
It suffices to prove $\sup_{B_R(x_0)}u\leq CM $ for every $M\geq \inf_{B_R(x_0)}u+\frac{\SetCentro{\eta R}{x_0}{}}{\e_P}+\delta$ for every $\delta>0$. Consider 
\begin{equation*}
\quad\tilde{u}:=\frac{u}{M}\in\K_{\Omega,\tilde{f}},\quad\text{where } \tilde{f}:=\frac{f}{M}
\end{equation*}
clearly $\inf_{B_R(x_0)}\tilde{u}\leq 1$ and $\mathcal{S}(B_{\eta R}(x_0),\tilde{f})\leq \e_P$, so that by Proposition \ref{Harnack u,f small} we find $\sup_{B_R(x_0)}u\leq CM$, hence
\begin{align*}
\sup_{B_R(x_0)}u &\leq C\biggl(\inf_{B_R(x_0)}u+\frac{\SetCentro{\eta R}{x_0}{}}{\e_P} +\delta \biggr)\\
&\leq \frac{C}{\e_P}\biggl(\inf_{B_r(x_0)}u+\SetCentro{\eta R}{x_0}{}+\delta \biggr).
\end{align*}
If we let $\delta\to 0^+$ we get the thesis.
\end{proof}

\section{Application to Grushin type operators}

Let $\Omega\subseteq\mathbb{R}^2$ be open (with respect to the Euclidean distance) and consider the measure space $(\Omega,\mathscr{L})$, where $\mathscr{L}$ is the Lebesgue measure, and the second order linear operator
\begin{equation}\label{def L}
L:=a_{11}(x_1,x_2)X^2+a_{22}(x_1,x_2)Y^2+2a_{12}(x_1,x_2)YX,
\end{equation}
where $a_{ij}:\Omega\to\mathbb{R}$ are measurable functions and
\begin{equation}\label{def Grushin vf}
X:=\partial_{x_1},\quad Y:=x_1\partial_{x_2}
\end{equation}
are Grushin vector fields. Moreover we assume there exist positive constants, called ellipticity constants, $0<\lambda\leq \Lambda$ such that for every $\xi=(\xi_1,\xi_2)\in \R^2$ we have
\begin{equation}\label{grushin operator ellipticity constant}
\lambda |\xi|^2\leq a_{11}\xi_1^2+a_{22}\xi_2+2a_{12}\xi_1\xi_2\leq\Lambda|\xi|^2.
\end{equation}

Using the theory presented in Section \ref{sect abstract H} we will be able to prove the Harnack inequality for non negative classical solution of equations of the type
\begin{equation*}\label{Lu=f}
Lu=x_1^2f
\end{equation*}
with $f\in \mathcal{F}(\Omega)=\{f:\Omega\to\mathbb{R}, \text{measurable and bounded}\}$. We highlight that this result is an improvement of the one obtained in \cite{AM} where the homogeneous case $Lu=0$ is considered, nevertheless we make use of many results and ideas there developed. 
In accordance with the notation of Section \ref{sect abstract H}, for every measurable set $A\subset\Omega$ we define
\begin{gather*}
\mathcal{S}_{\Omega}(A,f):=\diam{A}\|x_1f\|_{L^2(A)},\\
\mathcal{L}(\Omega)=\{ f\in \mathcal{F}(\Omega): \,
\mathcal{S}_{\Omega}(B_r(x),f)<+\infty, \,\, \forall B_r(x)\subseteq \Omega \}\\
\mathbb{K}_{\Omega, f}:=\{u\in C^2(\Omega)\cap C(\ol{\Omega}):\; Lu=x_1^2f,\;u\geq 0\},\, \forall f \in \mathcal{L}(\Omega) \\
%\mathbb{K}_{\Omega, f}(A):=\{u\in\mathbb{K}_{\Omega, f} \quad\text{whose domain contains }A\}\\
\end{gather*}
where $\diam{A}$ denotes the Euclidean diameter of $A$. Since the operator $L$ is linear and with no zero-order term,
if $u$ is a classical solution to $Lu=x_1^2f$, the function $\tau-\lambda u$ solves  $L(\tau-\lambda u)=-\lambda x_1^2f$, for every $\lambda,\;\tau\in\mathbb{R}$. So that the definition of $\mathbb{K}_{\Omega, f}$ is coherent with Definition \ref{def KOmega}.

\subsection{Grushin metric, sublevel sets and useful results}

We recall some useful and well known facts about the Carnot–Carathéodory metric $d_{CC}$ induced by vector fields \eqref{def Grushin vf} paying particular attention to the structure of balls $B_{CC}$ associated to this metric.

Using results proved by Franchi and Lanconelli in \cite{FL} it is possible to describe the structure of balls $B_{CC}$ by means of boxes
\begin{equation*}
\text{Box}(x,r):=]x_1-r,x_1+r[\;\times\;]x_2-r(r+|x_1|),x_2+r(r+|x_1|)[ .
\end{equation*}

\begin{thm}[Structure Theorem I, \cite{AM} Theorem 3.2]\label{thm structure 1}
There exists a structural constant $C_C>1$ such that 
\begin{equation*}
\text{Box}(x, C_C^{-1}r)\subset B_{CC}(x,r)\subset \text{Box}(x, C_Cr).
\end{equation*}
\end{thm}

Hence the following inequality holds for every $x\in\mathbb{R}^2$, $r>0$ and a suitable structural constant $C$. 
\begin{equation*}%\label{mis box bcc}
C^{-1}|\text{Box}(x,r)|\leq |B_{CC}(x,r)|\leq C|\text{Box}(x,r)|.
\end{equation*}
Here and throughout this section we denote by $|E|$ or the Lebesgue measure of a measurable set $E$.

We now define suitable sublevel sets of a specific functions $g$ and $h$ in which we are able to construct barriers to prove the critical density and the double ball property respectively. We construct these functions modifying the fundamental solution 
$\Gamma(x,0)=\bigl(x_1^4+4x_2^2\bigr)^{(2-Q)/4}$ with pole at the origin of the subelliptic Laplacian $X^2+Y^2$ where $	Q=3$ is the homogeneous dimension.
As in \cite{AM} we consider $\rho(x,y)=\bigl((x_1^2-y_1^2)^2+4(x_2-y_2)^2\bigr)^{1/4}$, for every $r>0$, $y=(y_1,y_2)\in\mathbb{R}^2$ and we define
\begin{equation*}
\tilde{g}_r(x,y)=
\begin{cases}
\rho(x,y)\quad &\text{if }|y_1|<r\\
\frac{1}{|y_1|}\rho^2(x,y)\quad &\text{if }|y_1|\geq r.
\end{cases}
\end{equation*}
We denote the sublevel sets of the function $\tilde{g}_r(\cdot,y)$ by
\begin{equation*}
\tilde{G}(y,r):=\{x\in\mathbb{R}^2:\tilde{g}_r(x,y)<r\}.
\end{equation*}

In order to avoid two zeros of $\tilde{g}_r$ in $\tilde{G}(y,r)$ we also define the function
\begin{equation*}
g_r(x,y)=
\begin{cases}
\rho(x,y)\quad &\text{if }|y_1|<r\\
\frac{1}{|y_1|}\rho^2(x,y)\quad &\text{if }|y_1|\geq r\text{ and }x_1y_1\geq 0\\
+\infty \quad &\text{if }|y_1|\geq r\text{ and }x_1y_1\leq 0
\end{cases}.
\end{equation*}
and consider its sublevel sets
\begin{equation*}
G(y,r):=\{x\in\mathbb{R}^2:g_r(x,y)<r\}.
\end{equation*}

Moreover, for every $r>0$ and $y\in\mathbb{R}^2$, we define
\begin{equation}\label{def: sigma}
\sigma(x,y)=\bigl((x_1^2-y_1^2)^2+ 2y_1^2(x_1-y_1)^2+4(x_2-y_2)^2\bigr)^{1/4}
\end{equation}

and
\begin{equation*}
h_r(x,y)=
\begin{cases}
\sigma(x,y)\quad &\text{if }|y_1|<r\\
\frac{1}{|y_1|}\sigma^2(x,y)\quad &\text{if }|y_1|\geq r.
\end{cases}
\end{equation*}
Sublevel sets of the function $h_r(\cdot,y)$ will be denoted by
\begin{equation*}
H(y,r):=\{x\in\mathbb{R}^2:h_r(x,y)<r\}.
\end{equation*}
We remark that, contrary to $\tilde{G}$, sublevel sets $H$ are always connected but not symmetric with respect to $\{x=0\}$.  

In the next subsection are extensively used Theorems below in which sublevels set $G(y,r)$ and $H(y,r)$ are compared with boxes $Box(y,r)$.  

\begin{thm}[Structure Theorem II, \cite{AM} Theorem 3.6]\label{thm structure Box G}
There exists a structural constant $C_G>1$ such that for every $y\in \mathbb{R}^2$ and $r>0$
\begin{equation*}
\text{Box}(x, C_G^{-1}r)\subset G(x,r)\subset \text{Box}(x, C_Gr).
\end{equation*}
\end{thm}

\begin{rmk}\label{rmk stima dal basso diam G}
Whenever $y=(y_1,0)\in\mathbb{R}^2$, from the Theorem above easily follows
\begin{gather}
c^{-1}r^2(r+|y_1|)\leq|G(y,r)|\leq cr^2(r+|y_1|)\label{area G}\\
\sup_{G(y,2r)}|x_1|\leq \tilde{c}(r+|y_1|)\label{sup x-1 su G}\\
C_M\max\{r,r(r+|y_1|)\}\geq\diam{G(y,r)}\geq C_m\max\{r,r(r+|y_1|)\}.\label{diam G}
\end{gather}
with $c,\tilde{c},C_M>1$ and $C_m>0$ structural constant. 
\end{rmk}

\begin{thm}[Structure Theorem III]\label{thm structure Box H}
There exists a structural constant $C_H>1$ such that for every $y\in \mathbb{R}^2$ and $r>0$
\begin{equation*}
\text{Box}(x, C^{-1}_Hr)\subset H(x,r)\subset \text{Box}(x, C_Hr).
\end{equation*}
\begin{proof}
To prove the inclusions we shall distinguish two case: $|y_1|<r$ and $|y_1|\geq r$. If we assume $|y_1|<r$ and $x\in \text{Box}(y,r/4)$ then
$$|x_1-y_1|<r/4,\quad |x_2-y_2|<r/4(r/4+|y_1|)<1/4(1/4+1)r^2 $$
and
\begin{align*}
g_r(x,y)^4&=\sigma^4(x,y)=(x_1-y_1)^2(x_1+y_1)^2+2y_1^2(x_1-y_1)^2+4(x_2-y_2)^2\\
&\leq (1/4)^2(|x_1|+|y_1|)^2r^2+4(1/4)^2r^4(1/4+1)^2+2(1/4)^2r^4\\
&\leq (1/4)^2(|x_1-y_1|+2|y_1|)^2r^2+1/4r^4(1/4+1)^2+2(1/4)^2r^4\\
&\leq (1/4)^2(1/4+2)^2r^4+1/4(1/4+1)^2r^4+2(1/4)^2r^4\\
&=((1/4)^2(1/4+2)^2+1/4(1/4+1)^2+2(1/4)^2)r^4\\
&=(213/256)r^4<r^4
\end{align*}
Hence $x\in H(y,r)$. Moreover if  $|y_1|<r$ and $x\in H(y,r)$ we have 
$$|x_1^2-y_1^2|<r^2,\quad 4|x_2-y_2|^2<r^4 $$
from which we get
$$|x_1-y_1|\;||x_1-y_1|-2|y_1||\;<r^2,\quad |x_2-y_2|<\frac{r^2}{2}\leq r(r+|y_1|).$$
Consequently if $|x_1-y_1|>2|y_1|$, adding $|y_1|^2$ to both sides of the first inequality above and taking the square root gives $$|x_1-y_1|-|y_1|<\sqrt{r^2+|y_1^2|}<2r$$ from which we easily deduce $x\in \text{Box}(y,3r)$. On the other hand if $|x_1-y_1|\leq 2|y_1|$ then $x\in \text{Box}(y,3r)$. Hence we have proved the thesis for $|y_1|<r$.\\
To prove the other case suppose $|y_1|\geq r$ and $x\in \text{Box}(y,r/4)$ then
$$|x_1-y_1|<r/4,\quad |x_2-y_2|<r/4(r/4+|y_1|)<1/4(1/4+1)|y_1|^2 $$
and
\begin{align*}
g_r(x,y)^2|y_1|^2&=\sigma^4(x,y)=(x_1-y_1)^2(x_1+y_1)^2+2y_1^2(x_1-y_1)^2+4(x_2-y_2)^2\\
&\leq (1/4)^2(|x_1|+|y_1|)^2r^2+4(1/4)^2r^2y_1^2(1/4+1)+2(1/4)^2r^2y_1^2\\
&\leq (1/4)^2(|x_1-y_1|+2|y_1|)^2r^2+1/4r^2y_1^2(1/4+1)+2(1/4)^2r^2y_1^2\\
&\leq (1/4)^2(1/4+2)^2r^2y_1^2+1/4r^2y_1^2(1/4+1)+2(1/4)^2r^2y_1^2\\
&=((1/4)^2(1/4+2)^2+1/4(1/4+1)^2+2(1/4)^2)r^2y_1^2\\
&=(213/256)r^2y_1^2<r^2y_1^2
\end{align*}
Hence $x\in H(y,r)$. Moreover if $x\in H(y,r)$ 
$$\sqrt{2}|y_1||x_1-y_1|<r|y_1|,\quad 4|x_2-y_2|^2<r^2|y_1|^2$$
consequently 
$$|x_1-y_1|<r,\quad |x_2-y_1|<\frac{r|y_1|}{2}<r(r+|y_1|),$$
hence $x\in\text{Box}(y,r)$ concluding the proof.
\end{proof}

\end{thm}

Theorem \ref{thm structure 1} is not sufficient to conclude that the ring condition (see Definition \ref{ring condition}) holds in the metric space $(\Omega,d_{CC},\mathscr{L})$. In \cite{AM} this problem is overcome introducing the new H\"older quasi distance 
\begin{equation*}
\tilde{d}(x,y):=|x_1-y_1|+\sqrt{x_1^2+y_1^2+4|x_2-y_2|}-\sqrt{x_1^2+y_1^2}
\end{equation*}
on the measure space $(\Omega,\mathscr{L})$. The triplet $(\Omega, \tilde{d},\mathscr{L})$ turns out to be a doubling quasi metric space satisfying the ring condition (\cite[Theorem 3.4]{AM}), moreover from the theorem below we deduce that $\tilde{d}$ is equivalent to the Carnot–Carathéodory metric $d_{CC}$.

For $r>0$ we denote by
\begin{equation*}
B(x,r):=\{y\in\mathbb{R}^2:\tilde{d}(x,y)<r\},
\end{equation*}
the quasi metric ball of center $x$ and radius $r$.

\begin{thm}[Structure Theorem IV, \cite{AM} Theorem 3.3]\label{thm structure Box B}
There exists a constant $C_B > 1$ such that
\begin{equation*}
\text{Box}(x, C_B^{-1}r)\subset B(x,r)\subset \text{Box}(x, C_Br) \quad\text{for every } y\in \mathbb{R}^2,\; r>0.
\end{equation*}
\end{thm}

%By Theorems \ref{thm structure 1} and \ref{thm structure 2} follows that for every $r>0$ and $x\in\mathbb{R}^2$
%\begin{equation}\label{inclusione Bcc G}
%B_{CC}(x,r)\subset \text{Box}(x, C_1r)\subset G(x, C_1C_2r)\subset \text{Box}(x,C_1C_2^2r)\subset B_{CC}(x,C_1^2C_2^2r).
%\end{equation}
%
%Similarly, Theorems \ref{thm structure 2} and \ref{thm structure 3} imply 
%\begin{equation}\label{inclusione B G}
%B(x,r)\subset \text{Box}(x, C_3r)\subset G(x, C_2C_3r)\subset \text{Box}(x,C_2^2C_3r)\subset B(x,C_2^2C_3^2r)
%\end{equation}
%for every $r>0$ and $x\in\mathbb{R}^2$.

It is well known that there exists a group of dilations $(\delta_t)_{t>0}$,
\begin{equation}\label{dilatations}
\delta_t:\mathbb{R}^2\to\mathbb{R}^2,\quad \delta_{t}(x)=(tx_1,t^2x_2)
\end{equation}
such that the vector fields $X_j$ are $\delta_t$-homogeneous of degree one , i.e. for every $u\in C^1(\mathbb{R}^2)$, $x\in \mathbb{R}^2$ and $t>0$
\begin{equation*}
X_j(u\circ\delta_t)(x)=t(X_ju)(\delta_t(x)).
\end{equation*}

A very useful tool for dealing with barrier functions is the weighted ABP maximum principle proved in \cite{AM}: %define the convex envelope of $u$ in $\Omega$ \begin{equation*}
%\Gamma(u)(x):=\sup_{w}\{w(x):w\leq u \text{ in }\Omega, w\text{ convex in }\Omega \}
%\end{equation*}
%and denote by $\Gamma_{u}$ the convex envelope of $-u^-$, where $u^-=\max\{0,-u\}$ as usual.

\begin{thm}[Weighted ABP Maximum Principle, \cite{AM} Theorem 2.5]\label{thm ABP}
Let $\Omega\subset\mathbb{R}^2$ be a bounded domain and assume $u\in C(\ol{\Omega})\cap C^2(\Omega)$, $u\geq 0$ on $\partial \Omega$ is a classical solution of $Lu(x)\leq x_1^2f(x)$ in $\Omega$ with $f$ bounded. Then there exists a positive structural constant $C$ such that 
\begin{equation*}
\sup_{\Omega}u^-\leq C\diam{\Omega}\biggl(\int_{\Omega\cap\{u=\Gamma_u\}}(x_1f^+)^2\;dx\biggr)^{\frac{1}{2}}.
\end{equation*}
Here $d=\diam{\Omega}$ is the Euclidean diameter of $\Omega$, $f^+(x):=\max\{f(x),0\}$, $\Gamma_u$ is the convex envelope of $-u^-(x):=-\max\{-u(x),0\}$ in a Euclidean ball of radius $2d$ containing $\Omega$ and $u\equiv 0$ outside $\Omega$.
\end{thm}
\subsection{Double Ball Property}

In this subsection we prove double ball property for sublevel sets $H(y,r)$, (see Theorem \ref{thm: DB in H}), and then extend it to balls $B(x,r)$ with the aid of Structure Theorems \ref{thm structure Box H} and \ref{thm structure Box B}. \\

\begin{lem}\label{lem: barrier DB}
Let $\Lambda,\;\lambda$ be the ellipticity constants and $\sigma$ the function defined in \eqref{def: sigma}. For $\alpha\leq 4-10\Lambda/\lambda$, The function $\phi(x)=\sigma^{\alpha}$ is a classical solution of $L\phi\geq 0$ in the set $\{\sigma>0\}$.
\end{lem}
\begin{proof}
For clarity reasons we compute separately first and second partial derivatives of $\sigma$:
\begin{align*}
&\sigma_{x_1}=\sigma^{-3}\bigl((x_1^2-y_1^2)x_1+y_1^2(x_1-y_1)\bigr)=\sigma^{-3}(x_1^3-y_1^3)\\
%\bigl((x_1+y_1)x_1+y_1^2\bigr) \\
&\sigma_{x_2}=2\sigma^{-3}(x_2-y_2)\\
&\sigma_{x_1x_1}=-3\sigma^{-7}\bigl(x_1^3-y_1^3\bigr)^2+3\sigma^{-3}x_1^2\\
%=-3\sigma^{-7}(x_1-y_1)^2\bigl((x_1+y_1)x_1+y_1^2\bigr)^2+3\sigma^{-3}x_1^2\\
&\sigma_{x_1x_2}=-3\sigma^{-7}\bigl(x_1^3-y_1^3\bigr)(x_2-y_2)\\
%=-3\sigma^{-7}(x_1-y_1)\bigl((x_1+y_1)x_1+y_1^2\bigr)(x_2-y_2)\\
&\sigma_{x_2x_2}=-3\sigma^{-7}(x_2-y_2)^2+2\sigma^{-3}.\\
\end{align*} 
Using calculation above it is straightforward to compute
\begin{align*}
L\phi &= \alpha\sigma^{\alpha-2}\bigl((\alpha-1)(a_{11}\sigma_{x_1}^2+2x_1a_{12}\sigma_{x_2}^2+x_1^2a_{22}\sigma_{x_2}^2)+\sigma L\sigma \bigr)\\
&=\alpha\sigma^{\alpha-2}\Bigl(a_{11}\bigl((\alpha-1)\sigma_{x_1}^2+\sigma\sigma_{x_1x_1}\bigr)+2x_1a_{12}\bigl((\alpha-1)\sigma_{x_1}\sigma_{x_2}+\sigma\sigma_{x_1x_2}\bigr)+x_1^2a_{22}\bigl((\alpha-1)\sigma_{x_2}^2+\sigma\sigma_{x_2x_2}\bigr)\Bigr)\\
&=\alpha\sigma^{\alpha-2}\biggl((\alpha-4)\sigma^{-6}\Bigl(a_{11}(x_1^3-y_1^3)^2+2a_{12}x_1(x_1^3-y_1^3)2(x_2-y_2)+a_{22}x_1^24(x_2-y_2)^2\Bigr)\\
&\qquad  \qquad+ x_1^2\sigma^{-2}(3a_{11}+2a_{22})\biggr)\\
&=\alpha\sigma^{\alpha-8}\bigl((\alpha-4)\gamma + x_1^2\sigma^{4}(3a_{11}+2a_{22})\bigr)
\end{align*}
where
$$\gamma:=a_{11}(x_1^3-y_1^3)^2+4a_{12}x_1(x_1^3-y_1^3)(x_2-y_2)+4a_{22}x_1^2(x_2-y_2)^2.$$
By \eqref{grushin operator ellipticity constant}, we get 
$\gamma\geq \lambda\Bigl((x_1^3-y_1^3)^2+4x_1^2(x_2-y_2)^2 \Bigr)$ and $3a_{11}+2a_{22}\leq 5\Lambda$, so that
\begin{align*}
L\phi &\geq\alpha\sigma^{\alpha-8}\biggl((\alpha-4)\lambda\Bigl((x_1-y_1)^2\bigl((x_1+y_1)x_1+y_1^2\bigr)^2+x_1^24(x_2-y_2)^2 \Bigr) + 5\Lambda x_1^2\sigma^{4}\biggr)\\
&=\alpha\sigma^{\alpha-8}\Bigl(\bigl((\alpha-4)\lambda+5\Lambda\bigr)(x_1^2-y_1^2)^2x_1^2+4\bigl((\alpha-4)\lambda+5\Lambda\bigr)(x_2-y_2)^2x_1^2+\\
&\qquad +y_1^2(x_1-y_1)^2\bigl(2\lambda(\alpha-4)(x_1+y_1)x_1+\lambda(\alpha-4)y_1^2+10\Lambda x_1^2\bigr)\Bigr)\\
&=\alpha\sigma^{\alpha-8}\Bigl(\bigl((\alpha-4)\lambda+5\Lambda\bigr)(x_1^2-y_1^2)^2x_1^2+4\bigl((\alpha-4)\lambda+5\Lambda\bigr)(x_2-y_2)^2x_1^2+\\
&\qquad +y_1^2(x_1-y_1)^2\bigl(2(\lambda(\alpha-4)+5\Lambda)x_1^2+2\lambda(\alpha-4)x_1y_1+\lambda(\alpha-4)y_1^2\bigr)\Bigr)\\
&=\alpha\sigma^{\alpha-8}\Bigl(\bigl((\alpha-4)\lambda+5\Lambda\bigr)(x_1^2-y_1^2)^2x_1^2+4\bigl((\alpha-4)\lambda+5\Lambda\bigr)(x_2-y_2)^2x_1^2+\\
&\qquad +y_1^2(x_1-y_1)^2\bigl((\lambda(\alpha-4)+10\Lambda)x_1^2+\lambda(\alpha-4)x_1^2+2\lambda(\alpha-4)x_1y_1+\lambda(\alpha-4)y_1^2\bigr)\Bigr)\\
\end{align*}
By the definition of $\alpha,$ we have $\alpha(\lambda(\alpha-4)+5\Lambda)\geq 0$ and $\alpha(\lambda(\alpha-4)+10\Lambda)\geq 0$, and we finally conclude
\begin{align*}
L\phi &\geq \lambda\alpha(\alpha-4)\sigma^{\alpha-8}y_1^2(x_1-y_1)^2\bigl(x_1^2+2x_1y_1+y_1^2\bigr)\geq 0.\\
\end{align*}

\end{proof}

\begin{thm}\label{thm barrier DB}
Let $\sigma$ be the function defined in \eqref{def: sigma} and $\alpha$ as in Lemma above. There exist positive structural constants $M_1$, $M_2$ and $0<\gamma<1$ such that, defined $R(y,r,3r):=H(y,3r)\setminus \ol{H(y,r)}$, for every $y\in\mathbb{R}^2$ and $r>0$ the function $\Phi:=M_2\sigma^\alpha-M_1$ satisfies
\begin{itemize}
\item $\Phi\in C^2(R(y,r,3r))\cap C(\ol{R(y,r,3r)})$,
\item $L\Phi\geq 0$ on $R(y,r,3r)$,
\item $\Phi \mid_{\partial H(y,3r)}=0$,
\item $\Phi \mid_{\partial H(y,r)}=1$,
\item $\inf_{R(y,r,2r)}\Phi\geq \gamma$.
\end{itemize}
\end{thm}
\begin{proof}
We choose $M_1$ and $M_2$ such that $\Phi\mid_{\partial H(y,3r)}=0$ and $\Phi\mid_{\partial H(y,r)}=1$ and we show that they are positive. Since $H(y,r)$ are defined as sublevel set of the function $h_r(x,y)$ and the definition of this function changes in case $|y_1|<r$ or $|y_1|\geq r$, we have to distinguish four cases.
\begin{itemize}
\item[Case I] $|y_1|<r$, we have\\
$M_1=\frac{3^{\alpha}}{1-3^{\alpha}}>0$, $M_2=\frac{1}{r^{\alpha}(1-3^{\alpha})}>0$ and we define $M_3:=\Phi\mid_{\partial H(y,2r)}=\frac{2^{\alpha}-3^{\alpha}}{1-3^{\alpha}}>0$\\
\item[Case II] $3r\leq |y_1|$, we have\\
$M_1=\frac{3^{\alpha/2}}{1-3^{\alpha/2}}>0$, $M_2=\frac{1}{(r|y_1|)^{\alpha/2}(1-3^{\alpha/2})}>0$ and we define $M_3:=\Phi\mid_{\partial H(y,2r)}=\frac{2^{\alpha/2}-3^{\alpha/2}}{1-3^{\alpha/2}}>0$
\item[Case III] $r\leq |y_1|< 2r$, we have\\
$M_1=\frac{3^{\alpha}}{(|y_1|/r)^{\alpha/2}-3^{\alpha}}>0$, $M_2=\frac{1}{(r|y_1|)^{\alpha/2}-(3r)^{\alpha}}>0$ \\
and we define $M_3:=\Phi\mid_{\partial H(y,2r)}=\frac{2^{\alpha}-3^{\alpha}}{(|y_1|/r)^{\alpha/2}-3^{\alpha}}\geq\frac{2^{\alpha}-3^{\alpha}}{1-3^{\alpha}}>0$
\item[Case IV] $2r\leq |y_1|<3r$, we have\\
$M_1=\frac{3^{\alpha}}{(|y_1|/r)^{\alpha/2}-3^{\alpha}}>0$, $M_2=\frac{1}{(r|y_1|)^{\alpha/2}-(3r)^{\alpha/2}}>0$ \\
and we define $M_3:=\Phi\mid_{\partial H(y,2r)}=\frac{(2|y_1|/r)^{\alpha/2}-3^{\alpha}}{(|y_1|/r)^{\alpha/2}-3^{\alpha}}\geq \frac{6^{\alpha/2-3^{\alpha}}}{2^{\alpha/2}-3^{\alpha}}>0$
\end{itemize}
If we define $\gamma:=\lbrace \min\frac{2^{\alpha}-3^{\alpha}}{1-3^{\alpha}}, \frac{2^{\alpha/2}-3^{\alpha/2}}{1-3^{\alpha/2}}, \frac{6^{\alpha /2}-3^{\alpha}}{3^{\alpha/2}-3^{\alpha}}\rbrace $ we have $\Phi\mid_{\partial H(y,2r)}=M_3>\gamma$, and since $\Phi$ is constant on the sets $\partial H(y,\rho)$ and decreasing with respect to $\rho>0$ get $\Phi\mid_{R(y,r,2r)}>\gamma$. Finally, by Lemma \ref{lem: barrier DB}, $L\Phi\geq 0$ on $R(y,r,3r)$, concluding the proof. 
\end{proof}

\begin{thm}\label{thm: DB in H}(Double ball property in $H(y,3r)$)
Let $C$ and $\gamma$ be as in Theorem \ref{thm ABP} and \ref{thm barrier DB} respectively and define $\e_0=\frac{\gamma}{2C}<1$. Then if $H(y,3r)\subset\Omega$ and $u$ is a non negative classical solution of $Lu=x_1^2f$ in $H(y,3r)$ satisfying 
\begin{equation*}
\inf_{H(y,r)} u\geq 1\quad\text{and}\quad\diam{H(y,3r)}\|x_1f\|_{L^2(H(y,3r))}<\e_0
\end{equation*}
we have
\begin{equation*}
\inf_{H(y,2r)} u\geq \delta
\end{equation*}
where $0<\delta<1$ is a constant depending on $\e_0$.
\end{thm}

\begin{proof}
Let $\Phi$ be the barrier function defined in Theorem  \ref{thm barrier DB} and consider $\omega=u-\Phi$. Since $\omega\in C^2(R(y,r,3r))\cap C(\ol{R(y,r,3r)})$, $\omega\geq 0$ on $\partial R(y,r,3r)$ and $L\omega\leq x_1^2f$ in $R(y,r,3r)$ we can apply the weighted ABP maximum principle Theorem \ref{thm ABP} to $\omega$
\begin{align*}
\sup_{R(y,r,3r)}((\Phi-u)^+)&=\sup_{R(y,r,3r)}\omega^-\\
&\leq C \text{diam}(R(y,r,3r))\int_{R(y,r,3r)}(x_1f^+)^2\;dx\\
&\leq C\e_0
\end{align*}
in particular
\begin{equation*}
\sup_{R(y,r,2r)}((\Phi-u)^+)\leq C \e_0.
\end{equation*}
Writing $-u=\Phi-u-\Phi$ and taking the supremum over $R(y,r,2r)$ we find
\begin{align*}
\sup_{R(y,r,2r)}(-u)&\leq \sup_{R(y,r,2r)}(\Phi-u)+\sup_{R(y,r,2r)}(-\Phi)\\
\intertext{i.e.}
-\inf_{R(y,r,2r)}u &\leq C \e_0-\inf_{R(y,r,2r)}\Phi.
\end{align*}
It suffices to recall that $\inf_{R(y,r,2r)}\Phi=\gamma$ and the definition of $\e_0$ to get
\begin{equation*}
\inf_{R(y,r,2r)}u\geq \frac{\gamma}{2},
\end{equation*}
moreover, since by hypotheses $\inf_{H(y,r)}u\geq 1$ the thesis is proved. 
\end{proof}

\begin{thm}\label{thm DB in B}(Double ball property in $B(y,\eta_{DB} r)$)
There exist structural constants $0<\e_{DB},\gamma<1$ and $\eta_{DB}>2$ such that if $u$ is a non negative classical solution of $Lu=x_1^2f$ in $B(y,\eta_{DB} r)$ satisfying 
\begin{equation*}
\inf_{B(y,r)} u\geq 1,\quad\text{and}\quad\diam{B(y,\eta_{DB} r)}\|x_1f\|_{L^2(B(y,\eta_{DB} r))}<\e_{DB}
\end{equation*}
then
\begin{equation*}
\inf_{B(y,2r)} u\geq \gamma.
\end{equation*}
More precisely we have $\eta_{DB}=12C^2$ with $C=C_HC_B$ and $C_H,C_B$ defined in Theorems \ref{thm structure Box H} and \ref{thm structure Box B}; $\e_{DB}=\e\delta^{p-1}$ and $\gamma=\delta^p$ where $p$ is chosen so that $2^{p-1}(C_HC_3)^{-1} \leq 2C_HC_3\leq 2^p(C_HC_3)^{-1} $; $\e$ and $\delta$ are defined in Theorem \ref{thm: DB in H}. 
\end{thm}
\begin{proof}
First of all we notice that the definition of $\eta_{DB}$ and $p$ imply $H(y,2^{p}C^{-1}3r)\subset B(y,\eta_{DB} r)$, consequently
\begin{equation}\label{eq norma f DB in B}
\diam{H(y,2^{k}C^{-1}3r)}\|x_1f\|_{L^2(H(y,2^{k}C^{-1}3r)}\leq \diam{B(y,\eta r)}\|x_1f\|_{L^2(B(y,\eta r))}<\e\delta^{p-1}\leq \e\delta^k 
\end{equation}
for $0\leq k\leq p-1$. Moreover, since $\inf_{B(y,r)}u\geq 1$, $H(y,C^{-1}r)\subset B(y,r)$ and \eqref{eq norma f DB in B} holds, we can apply the double ball property in $H(y,3C^{-1}r )$ (Theorem \ref{thm: DB in H}) obtaining $$\inf_{H(y,C^{-1}2r)}u\geq \delta. $$
In virtue of \eqref{eq norma f DB in B}  we repeatedly apply ($p-1$ times) Theorem \ref{thm: DB in H} to $\frac{u}{\delta^k}$ in $H(y,2^kC^{-1}3r) $ where $0<k\leq p-1$ and get 
$$\inf_{H(y,C^{-1}2^p r)}u\geq \delta^p. $$
Recalling the definition of $p$ and Theorems \ref{thm structure Box H} and \ref{thm structure Box B} we get the thesis. 
\end{proof}

\subsection{Critical Density}
In this subsection we prove critical density property for balls $B_r(y)$. Exactly as in \cite{AM} we obtain this property for balls $G((y_1,0),r)\subset\Omega$ centered on the $x$ axes and then extend the result to every set $G(y,r)\subset\Omega$ by a dilation and translation argument. Once one has obtained the critical density property for every set $G(y,r)\subset\Omega$, structure Theorems \ref{thm structure Box G} and \ref{thm structure Box B} easily imply the property on balls $B_r(x)\subset\Omega$.

\begin{lem}[\cite{AM}, Lemma 4.2]\label{lem bar}
There exist positive structural constants $\tilde{C}>0$ and $M>1$ such that for every $y=(y_1,0)\in\mathbb{R}^2$ and $r>0$ there is a $C^2$ function $\tilde{\phi}:\mathbb{R}^2\to\mathbb{R}$ such that
\begin{align*}
&\tilde{\phi}\geq 0,	& 	&\text{in }\mathbb{R}^2\setminus\tilde{G}(y,2r),\\
&\tilde{\phi}\leq 2,	&	&\text{in }\tilde{G}(y,r),\\
&\tilde{\phi}\geq -M,	&	&\text{in }\mathbb{R}^2\\
&L\tilde{\phi}(x)\leq \tilde{C}\frac{x_1^2}{r^2(r+|y_1|)^2}\zeta(x),	&	&\text{in }\mathbb{R}^2
\end{align*}
where $0\leq\zeta\leq 1$ is a continuous function in $\mathbb{R}^2$ with $\text{supp}\, \zeta\subset\ol{\tilde{G}(y,r)}$.
\end{lem}

\begin{thm}\label{thm rough critical density}
Define $\tilde{\e}=(2C)^{-1}$ where $C>0$ is the structural constant appearing in Theorem \ref{thm ABP}. Let $y=(y_1,0)$, $r>0$ and $u\in C^2(G(y,2r))\cap C(\ol{G(y,2r)})$ be a non negative solution of $Lu\leq x_1^2f$ in $G(y,2r)$ satisfying
\begin{equation*}
\inf_{G(y,r)}u \leq 1\quad\text{and}\quad\diam{G(y,2r)}\|x_1f\|_{L^2(G(y,2r))} <\tilde{\e} \label{norma f < eps} 
\end{equation*}
Then there exist $0<\nu<1$, depending on $\tilde{\e}$ and $M>1$ structural constant such that
\begin{gather}
|\{u\leq M\}\cap G(y,3r/2)| \geq \frac{\nu}{\max\bigl\{r+|y_1|,\frac{1}{r+|y_1|}\bigr\}}|G(y,3r/2)|.\label{tesi stima dal basso}
\end{gather}
\end{thm}

\begin{proof}
The function $w:=u+\tilde{\phi}$ with $\tilde{\phi}$ as in Lemma \ref{lem bar} satisfies $Lw\leq x_1^2\Bigl(f+ \zeta(x)\frac{\tilde{C}}{r^2(r+|y_1|)^2}\Bigr)$ in $G(y,2r)$, $w\geq 0$ on $\partial G(y,2r)$, and
\begin{equation*}%\label{eq inf rough cd}
\inf_{G(y,r)}w \leq \inf_{G(y,r)}u-2\leq -1.
\end{equation*}
Thus it is straightforward to apply weighted ABP maximum principle Theorem \ref{thm ABP} in $G(y,2r)$ and get
\begin{equation*}
\begin{split}
1&\leq \sup_{G(y,2r)}w^-\\
&\leq C\diam{G(y,2r)}\Biggl(\int_{\{w=\Gamma_w\}\cap G(y,2r)} \biggl(x_1f+\zeta(x)\frac{\tilde{C}x_1}{r^2(r+|y_1|)^2} \biggr)^2 \;dx\Biggr)^{1/2}\\
&\leq C\diam{G(y,2r)}\Biggl(\|x_1f\|_{L^2(G(y,2r))}+\frac{\tilde{C}}{r^2(r+|y_1|)^2}\biggl(\int_{\{w=\Gamma_w\}\cap G(y,2r)} \Bigl(\zeta(x)x_1\Bigr)^2 \;dx\biggr)^{1/2}\Biggr)\\
&\leq C\tilde{\e}+C\tilde{C}\frac{\diam{G(y,2r)}}{r^2(r+|y_1|)^2}\biggl(\int_{\{w=\Gamma_w\}\cap G(y,2r)} \Bigl(\zeta(x)x_1\Bigr)^2 \;dx\biggr)^{1/2}\\
&\leq C\tilde{\e}+ C\tilde{C}C_M\max\{r,r(r+|y_1|)\}\frac{\tilde{c}(r+|y_1|)}{r^2(r+|y_1|)^2}\Bigl(\int_{\{w=\Gamma_w\}\cap G(y,2r)}\zeta^2\;dx \Bigr)^{1/2}.
\end{split}
\end{equation*}
Last estimate takes into account Remark \ref{rmk stima dal basso diam G}. Moreover by Lemma \ref{lem bar}, $0\leq \zeta\leq 1$ and $\text{supp } \zeta\subset \ol{\tilde{G}(y,r)}$ so that keeping in mind \eqref{area G}
\begin{align*}
1-C\tilde{\e}&\leq K\frac{\max\{1,r+|y_1|\}}{r(r+|y_1|)}|{\{w=\Gamma_w\}\cap G(y,2r)}\cap\tilde{G}(y,r)|^{1/2}\\
&\leq K\frac{\max\{1,r+|y_1|\}}{r(r+|y_1|)}r(r+|y_1|)^{1/2}\frac{|{\{w=\Gamma_w\}\cap G(y,2r)}\cap\tilde{G}(y,r)|^{1/2}}{|G(y,r)|^{1/2}}\\
&\leq K\frac{\max\{1,r+|y_1|\}}{(r+|y_1|)^{1/2}}\frac{|{\{w=\Gamma_w\}\cap G(y,2r)}\cap\tilde{G}(y,r)|^{1/2}}{|G(y,r)|^{1/2}}.
\end{align*}
Where $K>1$ is a structural constant. Since $w=\Gamma_w$ implies $w\leq 0$ and consequently $u\leq -\phi\leq M $, recalling the definition of $\tilde{\e}$ we obtain 
\begin{equation*}%\label{eq rough cd}
|{\{u\leq M\}\cap G(y,2r)}\cap\tilde{G}(y,r)|^{1/2} \geq \frac{1}{2K} \frac{|G(y,r)|^{1/2}}{\max\{(r+|y_1|)^{-1/2},(r+|y_1|)^{1/2}\}}.
\end{equation*}
Now the proof proceeds exactly as in \cite[Theorem 5.1]{AM}.\\
\end{proof}

We use the same dilations and translations arguments developed in \cite[Theorem 5.2]{AM}, to extend Theorem \ref{thm rough critical density} to every sublevel set $G(y,r)$ and improve constant in \eqref{tesi stima dal basso}.

\begin{thm}\label{thm negaz CD in G}
Define $\e_0=\tilde{\e}\frac{C_m}{8C_M }$. Let $u\in C^2(G(y,2r))\cap C(\ol{G(y,2r)})$ with $y\in\mathbb{R}^2$, $r>0$ be a non negative solution to $Lu\leq x_1^2f$ in $G(y,2r)$ satisfying  
\begin{equation*}
\inf_{G(y,r)}u\leq 1,\quad \text{and} \quad \diam{G(y,2r)}\|x_1f\|_{L^2(G(y,2r))} <\e_0
\end{equation*}
then there exist structural constants $0<\e<1$, depending on $\e_0$ and $M>1$ such that 
\begin{equation*}
|\{u\leq M\}\cap G(y,3r/2 )|\geq \e|G(y,3r/2)|.
\end{equation*}
Here $C_m$,$C_M$ are the structural constants appearing in \eqref{diam G} and $\tilde{\e}$ is as in the statement of Theorem \ref{thm rough critical density}.
\end{thm}
\begin{proof}
The proof is organized in four steps, at each step we prove the critical density property for a larger family of sets $G(y,r)$.
\begin{itemize}
\item[STEP I] Fix $r=1$, $y_1\in [-1,1]$ and $y_2=0$. Applying Theorem \ref{thm rough critical density}, we find that \eqref{tesi stima dal basso} holds true with $\max\{(r+|y_1|)^{1/2},(r+|y_1|)^{-1/2}\}$ replaced by $\sqrt{2}$ i.e.
$$|\{u\leq M\}\cap G(y,3/2 )|\geq \frac{\nu}{2} |G(y,3/2)|.$$
\item[STEP II]Fix $r\geq 0$, $|y_1|\leq r$ and $y_2\in \mathbb{R}$. Keeping in mind \eqref{dilatations} we introduce the change of variables
\begin{equation}\label{def T(x)}
T(x)=T(x_1,x_2):=(rx_1,y_2+r^2x_2),
\end{equation}
obtaining
\begin{align*}
X\tilde{u}(x)&=rXu(T(x)),\\
Y\tilde{u}(x)&=rYu(T(x)),\\
YX\tilde{u}(x)&=r^2YXu(T(x))
\end{align*}
where $\tilde{u}(x):=u(T(x))$. Moreover $T(x)\in G(y,r)$ if and only if $x\in G((y_1/r,0),1)$. The new operator $\tilde{L}:=\tilde{a_{11}}X^2+2\tilde{a}_{12}YX+\tilde{a}_{22}Y^2$, with $\tilde{a}_{i,j}(x):=a_{i,j}(T(x))$, $i\leq j\in\{1,2\}$, is of the type \eqref{def L} with ellipticity constants $\Lambda>\lambda$ and applied to $\tilde{u}$ gives
$$\tilde{L}\tilde{u}=r^2Lu(T(x))\leq r^2(T(x))_1^2f(T(x))=r^4x_1^2f(T(x))=x_1^2\tilde{f}(x)$$
with $\tilde{f}(x):=r^4f(T(x))$. We claim that $\tilde{u}$ satisfies the hypotheses of Theorem \ref{thm rough critical density} on $G((y_1/r,0),2)$. The only not obviously satisfied requirement  is 
\begin{equation}\label{req 1}
\diam{G((y_1/r,0),2)}\|x_1\tilde{f}(x)\|_{L^2(G((y_1/r,0),2))}<\tilde{\e}.
\end{equation}
Recalling \eqref{diam G} and changing coordinates ($x_1=\xi_1/r$, $x_2=(\xi_2-y_2)/r^2$) we have
\begin{equation}\label{stima 1 step 2}
\begin{split}
\diam{G((y_1/r,0),2)}\|x_1\tilde{f}(x)\|_{L^2(G((y_1/r,0),2))}&\leq 4C_M\max\Bigl\{1,1+\frac{|y_1|}{r}\Bigr\}\Bigl(\int_{G((y_1/r,0),2)}(x_1\tilde{f}(x))^2\dd{x}\Bigr)^{1/2}\\
&\leq 4C_M\Bigl(1+\frac{|y_1|}{r}\Bigr)\Bigl(\int_{G(y,2r)}(r^3\xi_1 f(\xi))^2\frac{1}{r^3}\dd{\xi}\Bigr)^{1/2}\\
&\leq 4C_M\Bigl(1+\frac{|y_1|}{r}\Bigr)r^{3/2}\|\xi_1 f(\xi)\|_{L^2(G(y,2r))}\\
&\leq 4C_M(r+|y_1|)r^{1/2}\|\xi_1 f(\xi)\|_{L^2(G(y,2r))}\\
\end{split}
\end{equation}
on the other hand by hypothesis and \eqref{diam G}
\begin{align*}
\tilde{\e}\frac{C_m}{8C_M}&>\diam{G(y,2r)}\|\xi_1 f(\xi)\|_{L^2(G(y,2r))}\\
&\geq C_m \max\{r,r(r+|y_1|)\}\|\xi_1 f(\xi)\|_{L^2(G(y,2r))}
\end{align*}
so that 
\begin{equation}\label{stima 2 step 2}
\tilde{\e}\geq 8C_M r\max\{1,r+|y_1|\}\|\xi_1 f(\xi)\|_{L^2(G(y,2r))}.
\end{equation}
Now, if $r\geq 1$, it is straightforward  to concatenate inequalities \eqref{stima 1 step 2} and \eqref{stima 2 step 2} to obtain \eqref{req 1}. Otherwise if $r\leq 1$ and $|y_1|\leq r$, again combining \eqref{stima 1 step 2} and \eqref{stima 2 step 2} we find
\begin{align*}
\diam{G((y_1/r,0),2)}\|x_1\tilde{f}(x)\|_{L^2(G((y_1/r,0),2))}
&\leq 4C_M(r+|y_1|)r^{1/2}\|\xi_1 f(\xi)\|_{L^2(G(y,2r))}\\
&\leq 8C_Mr^{3/2}\|\xi_1 f(\xi)\|_{L^2(G(y,2r))}\\
&\leq 8C_M r\max\{1,r+|y_1|\}\|\xi_1 f(\xi)\|_{L^2(G(y,2r))}\\&\leq \tilde{\e}.
\end{align*}
Hence $\tilde{u}$ satisfies hypotheses of Theorem \ref{thm rough critical density} with $y_1/r\in [-1,1]$, so that, by STEP I we get 
\begin{align*}
|\{\tilde{u}\leq M\}\cap G((y_1/r,0),3/2)|\geq \frac{\nu}{2}|G((y_1/r,0),3/2)|
\end{align*}
and by Theorem \ref{thm structure Box G} we conclude
\begin{align*}
|\{u\leq M\}\cap G(y,3r/2)|&=|T\bigr(\{\tilde{u}\leq M\}\cap G((y_1/r,0),3/2)\bigl)|\\
&=r^3|\{\tilde{u}\leq M\}\cap G((y_1/r,0),3/2) |\\
&\geq r^3  \frac{\nu}{2}|G((y_1/r,0),3/2)|\\
&=  \frac{\nu}{2}|G(y,3r/2)|.
\end{align*}
\item[STEP III] Fix $r>0$, $|y_1|=1$ and $y_2=0$. If $r\geq 1$ we apply STEP II, otherwise, in case $0<r<1$ we apply Theorem \ref{thm rough critical density} and take into account that $\max\biggl\{r+|y_1|,\frac{1}{r+|y_1|}\biggr\}<2$, so that
\begin{equation*}
|\{u\leq M\}\cap G((y_1,0),3/2r)|> \frac{\nu}{2}|G((y_1,0),3/2r)|.
\end{equation*}
\item[STEP IV] Fix $r>0$, $y_1,y_2\in \mathbb{R}$. If $|y_1|\leq r$ we use STEP II. If $|y_1|>r$ apply the change of variable defined in the second step with $r$ replaced by $|y_1|$ in \eqref{def T(x)}. Again we want to make use of Theorem \ref{thm rough critical density} and again the only hypothesis we must check is
\begin{equation}\label{reqirement step 4} 
\diam{G^*}\|x_1\tilde{f}(x)\|_{L^2(G^*)}<\tilde{\e},
\end{equation}
where $\tilde{f}(x):=|y_1|^4f(T(x))$ and $G^*=G((y_1/|y_1|,0),2r/|y_1|)$. \\
Recalling \eqref{diam G} and changing coordinates ($x_1=\xi_1/|y_1|$, $x_2=(\xi_2-y_2)/|y_1|^2$) we have
\begin{equation}\label{stima 1 step 4}
\begin{split}
\diam{G^*}\|x_1\tilde{f}(x)\|_{L^2(G^*)}&\leq 4C_M\max\Biggl\{\frac{r}{|y_1|},\frac{r}{|y_1|}\biggl(1+\frac{r}{|y_1|}\biggr)\Biggr\}\biggl(\int_{G^*}(x_1\tilde{f}(x))^2\dd{x}\biggr)^{1/2}\\
&\leq \frac{4C_Mr}{|y_1|^{1/2}}\max\{|y_1|,|y_1|+r\}\|\xi_1 f(\xi)\|_{L^2(G(y,2r))}\\
&\leq \frac{4C_Mr}{|y_1|^{1/2}}(|y_1|+r)\|\xi_1 f(\xi)\|_{L^2(G(y,2r))}\\
\end{split}
\end{equation}
If $|y_1|\geq 1$, it is straightforward to concatenate \eqref{stima 1 step 4} and \eqref{stima 2 step 2} obtaining 
\eqref{reqirement step 4}. Otherwise if $0<r<|y_1|< 1$, we estimate the right hand side of \eqref{stima 1 step 4} as follows
\begin{equation}\label{stima 1 step 4 improved}
\begin{split}
\diam{G^*}\|x_1\tilde{f}(x)\|_{L^2(G^*)}&\leq \frac{4C_M}{|y_1|^{1/2}}r(|y_1|+r)\|\xi_1 f(\xi)\|_{L^2(G(y,2r))}\\
&\leq \frac{8C_M}{|y_1|^{1/2}}r|y_1|\|\xi_1 f(\xi)\|_{L^2(G(y,2r))}\\
\end{split}
\end{equation}
and then concatenate \eqref{stima 1 step 4 improved} with \eqref{stima 2 step 2} and get the desired estimate.\\
Hence $\tilde{u}$ satisfies hypotheses of Theorem \ref{thm rough critical density} and using the third step we find 
\begin{align*}
|\{u\leq M\}\cap G(y,3r/2)|&=|T\bigl(\{\tilde{u}\leq M\}\cap G((y_1/|y_1|,0),3r/(2|y_1|))\bigr)|\\
&=|y_1|^3|\{\tilde{u}\leq M\}\cap G((y_1/|y_1|,0),3r/(2|y_1|))|\\
&\geq \frac{\nu}{2} |y_1|^3|G((y_1/|y_1|,0),3r/(2|y_1|))|\\
&\geq \frac{\nu}{2}|G(y,3r/2)|.
\end{align*}
\end{itemize}
\end{proof}

\begin{thm}\label{thm negaz CD in B}
Define $\e_0$ as in Theorem \ref{thm negaz CD in G}. Let $u\in C^2(B(y,2R))\cap C(\ol{B(y,2R)})$ with $y\in\mathbb{R}^2$, $R>0$ be a non negative solution to $Lu\leq x_1^2f$ in $B(y,2R)$ satisfying
\begin{equation*}
\inf_{B(y,R/2C^2)}u\leq 1, \quad\text{and}\quad \diam{B(y,2R)}\|x_1f\|_{L^2(B(y,2R))} <\e_0.
\end{equation*}
Then there exist structural constants $0<\nu<1$, depending on $\e_0$ and $M>1$ such that 
\begin{equation*}
|\{u\leq M\}\cap B(y,R)|\geq \nu|B(y,R)|.
\end{equation*}
Here $\nu=\e C^{-4}$ with $\e$ defined as in statement of Theorem \ref{thm negaz CD in G} and $C=C_BC_G$ with $C_G$ and $C_B$ the constants in Theorem \ref{thm structure Box G} and \ref{thm structure Box B} respectively.
\end{thm}
\begin{proof}
First of all we define $r:=2R/3C$, since $B(y,R/2C^2)\subset G(y,3r/4)\subset G(y,3r/2)$ and $G(y,4r/3)\subset B(y,2R)$ imply respectively
\begin{equation*}
\inf_{G(y,3r/2)}u\leq \inf_{B(y,R/2C^2)}u\leq 1\quad\text{and}\quad \diam{G(y,4r/3)}\|x_1f\|{L^2(G(y,4r/3))}<\e_0
\end{equation*}
by Theorem \ref{thm negaz CD in G} there exist structural constants $0<\e<1$ and $M>1$ such that 
\begin{equation*}
|\{u\leq M\}\cap G(y,r )|\geq \e|G(y,r)|.
\end{equation*}
With the aid of Theorem \ref{thm structure Box G} and \ref{thm structure Box B} we estimate from below the right hand side by
\begin{align*}
\e|G(y,r)|&\geq \e C_G^{-1}r\bigr(|y_1|+C_G^{-1}r\bigl)\\
&= \e(C_G C_B)^{-2}C_B R\bigr(|y_1|+(C_G C_B)^{-2}C_B R\bigl)\\
&\geq \e (C_G C_B)^{-4}C_B R\bigr(|y_1|+C_B R\bigl)\\
&\geq \e C^{-4}|\text{Box}(y,C_B R)|\\
&\geq \e C^{-4}|B(y,R)|.
\end{align*}
concluding the proof.
\end{proof}

Provided that we invert the relation of dependence between $\nu$ and $\e_0$, the negation of Theorem above and the double ball property give
\begin{thm}\label{thm CD in B}
There exist structural constants $\eta_{CD}, M>1$ and $0<\nu, c, \e_{CD}<1$ such that if, $u\in C^2(B(y,\eta_{CD} R))\cap C(\ol{B(y,\eta_{CD} R)})$ with $y\in\mathbb{R}^2$, $R>0$ is a non negative solution to $Lu\leq x_1^2f$ in $B(y,\eta_{CD} R)$ satisfying 
\begin{equation*}
|\{u> M\}\cap B(y,R)|>(1-\nu)|B(y,R)|
\end{equation*}
then 
\begin{equation*}
\inf_{B(y,R)}u> c \quad\text{or }\quad \diam{B(y,\eta_{CD} R)}\|x_1f\|_{L^2(B(y,\eta_{CD} R))} \geq\e_{CD}.
\end{equation*}
More precisely, $\eta_{CD}=2\eta_{DB}$, $\e_{CD}=\min\{\gamma^p\e_{DB},\e_0\}$, $c=\gamma^p$ where $\gamma$, $\e_{DB}$, $\eta_{DB}$ are the constants defined in Theorem \ref{thm DB in B}, $M$, $\nu$, $\e_0$ are as in Theorem \ref{thm negaz CD in B}, and  $p\in\mathbb{N}$ is chosen so that $2^{p-1}>(C_GC_B)^2>2^{p-2}$  with $C_G$ and $C_B$ the constants in Theorems \ref{thm structure Box G} and \ref{thm structure Box B} respectively. 
\end{thm}
\begin{proof}
The negation of Theorem \ref{thm negaz CD in B} says that if 
\begin{equation*}
|\{u\leq M\}\cap B(y,R)\cap B(y,2 R)|< \nu|B(y,R)|
\end{equation*}
i.e.
\begin{equation*}
|\{u> M\}\cap B(y,R)|>(1-\nu)|B(y,R)|
\end{equation*}
then
\begin{equation*}
\inf_{B(y,R/2(C_GC_B)^2)}u> 1 \quad\text{or }\quad \diam{B(y,2 R)}\|x_1f\|_{L^2(B(y,2 R))} \geq\e_0.
\end{equation*}
If the second inequality holds or $\diam{B(y,\eta_{CD} R)}\|x_1f\|_{L^2(B(y,\eta{CD} R))} \geq\e_{CD}$ there is nothing to prove. Otherwise since by the definition of $p$, $2^p/2(C_GC_B)^2<\eta_{CB}$, for every $1\leq k\leq p$ we have
\begin{equation*}
\diam{B(y,R2^k/2(C_GC_B)^2)}\|x_1f\|_{L^2(B(y,R2^k/2(C_GC_B)^2))}\leq\diam{B(y,\eta{CD} R)}\|x_1f\|_{L^2(B(y,\eta{CD} R))} \leq\e_{CD}
\end{equation*}
we can repeatedly apply the double ball property $p$ times in $B(y,R/2(C_GC_B)^2)$ obtaining
\begin{equation*}
\inf_{B(y,R)}u\geq \inf_{B(y,2^pR/2(C_GC_B)^2)}u\geq \gamma^p.
\end{equation*}
\end{proof}

\subsection{Harnack inequality}
We need few more remarks to straightforwardly apply  the theory developed in Section \ref{sect abstract H}.
First of all we notice that  all structural constants in Theorem \ref{thm DB in B} and \ref{thm CD in B} are independent of the right hand side $f.$
%and it suffices to rename $f_{\lambda}:=\lambda f$,  and apply Theorem \ref{thm DB in B} and \ref{thm CD in B} to $u_{\lambda}$. 
This and the ring condition shown in \cite[Theorem 3.4]{AM} allow us to apply Theorems \ref{teo: power decay} and \ref{Harnack} to the family of function $\mathbb{K}_{\Omega,f}$ and get the following scale invariant Harnack inequality for the quasi metric balls $B$.

\begin{thm}
There exist structural constants $C, \eta>1$ such that if $u\in C^2(\Omega)\cap C(\ol{\Omega})$ is a non negative solution of $Lu=x_1^2f$ in $\Omega$ then 
\begin{equation*}
\sup_{B(y,r)}u\leq C\left(\inf_{B(y,r)}u+\diam{B(y,\eta r)}\|x_1f\|_{L^2(B(y,\eta r))}\right)
\end{equation*}
for every $B(y,\eta r)\subset\Omega$.
\end{thm}

Harnack inequality for Carnot–Carathéodory metric ball $B_{CC}$ follows easily from the theorem above and the two structure Theorems \ref{thm structure 1} and \ref{thm structure Box B}.

\section{Application to X-elliptic operators}

Let $(X_1,\dots,X_m)$ be a family of locally Lipschitz vector fields with coefficients defined in $\mathbb{R}^N$.\\
We consider linear second order differential operators of the type

\begin{equation*}\label{X operator 1}
Lu=\sum_{i,j=1}^N\partial_i(b_{ij}\partial_ju)+\sum_{i=1}^N b_i\partial_iu
\end{equation*}

where $b_{ij},b_i$ are measurable functions and $B=\{b_{ij}\}_{i,j=1\dots N}$ is a symmetric matrix. Moreover we assume $L$ to be uniformly X-elliptic in a bounded open set $\Omega\subset \mathbb{R}^N$ in the sense of \cite{GL} i.e.

\begin{Def}
We say that the operator $L$ is uniformly X-elliptic in an open subset $\Omega\subset \mathbb{R}^N$ if there exist positive constants $\lambda\leq\Lambda$ and a non negative function $\gamma$ such that
\begin{gather*}
\lambda \sum_{j=1}^m\langle X_j(x),\xi\rangle^2\leq \langle B(x)\xi,\xi\rangle\leq \Lambda\sum_{j=1}^m \langle X_j(x),\xi\rangle^2\quad\text{ for every } x\in \Omega\text{ and }\xi\in\mathbb{R}^N\\
\langle b(x),\xi\rangle^2\leq \gamma^2(x)\sum_{j=1}^m\langle X_j(x),\xi\rangle^2 \quad\text{ for every } x\in \Omega\text{ and }\xi\in\mathbb{R}^N
\end{gather*}
here we use the notation $b=(b_1,\dots, b_N)$ and $\langle\cdot,\cdot\rangle$ is the standard inner product in $\mathbb{R}^N$.
\end{Def}  

Applying the abstract theory developed in Section \ref{sect abstract H} we prove the non homogeneous Harnack inequality for weak solution (in the sense specified in \cite{U} and Definition \ref{def weak solutions}) of the equation
\begin{equation}\label{non omog X elliptic}
Lu=g+\sum_{i=1}^N\partial_i f_i
\end{equation}
where $f=(f_1,\dots,f_N)$ is a measurable function such that there exists a non negative function $\gamma_0$ satisfying 
\begin{equation*}
\langle f(x),\xi\rangle^2\leq \gamma_0^2(x)\sum_{j=1}^m\langle X_j(x),\xi\rangle^2 \text{ for every }x\in\Omega, \xi\in\mathbb{R}^N.
\end{equation*}
We remark that the operator $L$ here considered is a simplified version of the one studied by Gutiérrez, Lanconelli and Uguzzoni in \cite{GL} and \cite{U} where they prove Harnack inequality using an adapted Moser's iteration  technique. Authors in  \cite{GL} require the dilation invariance of vector fields, while in \cite{U} this assumption is removed. We assume and list below the same hypotheses considered by Uguzzoni in \cite{U}.

\begin{itemize}

\item The Carnot–Carathéodory distance $d$ related to the family of vectors fields $X$ is well defined and continuous with respect to the Euclidean topology.

\item (Doubling condition) $(\mathbb{R}^N,d,\mu)$ is a metric space satisfying: for each compact set $K\subset Y$ there exist positive constants $C_D>1$ and $R_0>0$ such that
\begin{equation}\label{doubling X elliptic}
0<|B_{2r}(x)|\leq C_D|B_{r}(x)|
\end{equation}
for every $d$-ball $B_r(x)$ with $x\in K$ and $r\leq R_0$. Hereafter $|E|$ denotes the Lebesgue measure of a measurable set $E$.

\item (Poincaré inequality) For each compact set $K\subset \mathbb{R}^N$ there exists a positive constant $C$ such that
\begin{equation*}
\fint_{B_r}|u-u_r|\;dx\leq Cr\fint_{B_{2r}}|Xu|\;dx
\end{equation*}
 for every $C^1$ function $u$ and for every $d$-ball $B_r(x)$ with $x\in K$ and $r\leq R_0$. Here $\fint_{B_r}:=\frac{1}{|B_r|}\int_{B_r} u \;dx$.
 
 \item Set $Q:=\log_2 C_D$ and $p>\frac{Q}{2}$, then
 \begin{equation*}
 \gamma,\gamma_0\in L^{2p}(\Omega)\quad \text{and}\quad g\in L^p(\Omega).
 \end{equation*}
 Notice that, enlarging $C_D$ in \eqref{doubling X elliptic} if needed, we can always assume (and we do this) $Q>2$.
 
 \item (Sobolev inequality) Previous assumptions on the family of vector fields $X$ imply 
 \begin{equation*}
 \|u\|_{\frac{2Q}{Q-2}}\leq C(D)\|Xu\|_2\quad\text{for every }u\in C_0^1(D)
 \end{equation*}
 for every open set $D$ with sufficiently small diameter and closure contained in the interior of $K_0$. Here $K_0\subset\mathbb{R}^N$ is a fixed compact set whose interior contains the closure of $\Omega$. (For a deeper discussion on this result see \cite{FLW}, \cite{GN}, \cite{HK}).
 
 \item (Reverse Doubling) The $d$-diameter of $\Omega$ is small enough to have the reverse doubling property (see Propositions 2.9 and 2.10 in \cite{DGL} ) i.e. for every $B_{2r}(x)\subset \Omega$ there exists a constant $0<\delta<1$ independent of $r$ such that
\begin{equation*}
|B_r(x)|\leq\delta |B_{2r}(x)|.
\end{equation*}
 \end{itemize}
 
 We highlight that the $d-$diameter of the set $\Omega$ is required to be sufficiently small and the doubling and the Poincaré inequality are local conditions, so the  Harnack inequality is obtained for balls with small enough radius and the constants appearing depend on the compact set fixed. Very recently, in \cite{BB} Battaglia and Bonfiglioli obtained an invariant non homogeneus Harnack inequality for solutions of a class of sub-elliptic operators in divergence form under global doubling and Poincaré assumptions but no restrictions on the diameter of the set $\Omega$.

Hereafter $Xu$ denotes the $X$-gradient of $u$:
\begin{equation*}
Xu=(X_1u,\dots,X_m u),
\end{equation*}
$K_0$ is a fixed compact set containing $\ol{\Omega}$, $r_0=r_0(K_0)>0$ is a constant such that for every $r<r_0$ the Sobolev inequality holds in $B_{4r}\subset\Omega.$ If 
$D$ is a bounded domain supporting the Sobolev inequality we define $W^1_0(D,X)$ the closure of $C^1_0(D)$ with respect to the norm $\|u\|=\|Xu\|_{L^2}$ and $W^1(D,X)=\{u\in L^2(D): Xu\in L^2(D)\}$. Moreover from the Sobolev inequality follows $W_0^1(D,X)\subset L^{\frac{2Q}{Q-2}}(D)$.

We consider the bilinear form
\begin{equation*}\label{X operator 2}
\mathfrak{L}(u,v)=\int_D\langle B\nabla u,\nabla v\rangle-\langle b,\nabla u\rangle v \;dx\quad \text{for } u \in C^1(D),\; v\in C^1_0(D)
\end{equation*}
and the linear functional
\begin{equation*}
\mathfrak{F}(v)=\int_D\langle f,\nabla v\rangle -gv\;dx\quad \text{for } v\in C^1_0(D).
\end{equation*}

From the uniform $X$-ellipticity of $L$, the Sobolev inequality and assumptions on $f$ and $g$, Gutiérrez, Lanconelli in \cite{GL} and Uguzzoni in \cite{U} show respectively that $\mathfrak{L}$ can be extended continuously  to $W^1(D,X)\cap L^r(D)\times W_0^1(D,X)$, where $\frac{1}{r}=\frac{1}{2}+\frac{1}{2p}$ and $\mathfrak{F}$ can be extended continuously to $W_0^1(D,X)$. This gives meaning to the following notion of weak solution.

\begin{Def}(Weak solutions)\label{def weak solutions} We say that a function $u\in W^1_{loc}(\Omega, X)$ is a weak subsolution (resp. supersolution) to $Lu=g+\sum_{i=1}^N\partial_i f_i$ in $\Omega$ if, for every domain $D$, supporting the Sobolev inequality, with closure contained in $\Omega$, we have
\begin{equation*}
\mathfrak{L}(u,v)\leq (\text{resp.}\geq) \mathfrak{F}(v) \quad\text{for every } v\in W_0^1(D,X).
\end{equation*}
We say that $u$ is a solution if it is both a super and a subsolution.
\end{Def}

The main step toward the proof of Harnack inequality is to show that if we define 
\begin{equation*}
\mathcal{S}_{\Omega}\Biggl(B_{R}(x_0),\;g+\sum_{i=1}^N\partial_i f_i\Biggr):=R^{\delta}\|\gamma_0\|_{L^{2p}(\Omega)}+R^{2\delta}\|g\|_{L^p(\Omega)}
\end{equation*}
for $B_{R}(x_0)\subseteq\Omega$, $R<r_0$ with $\delta=1-\frac{Q}{2p}$, then the family
\begin{gather*}
\K_{\Omega, g+\sum_{i=1}^N\partial_i f_i}=\{u\in W^1_{loc}(\Omega, X): u\text{ is a non negative weak soluton to \eqref{non omog X elliptic}}\}
\end{gather*}
has the $\nu$ critical density property for every $0<\nu<1$. We prove it with the aid of three Lemmas.

\begin{lem}[Local boundedness]\label{Local boundness Lemma}
Let $u$ be a solution to \eqref{non omog X elliptic} in $\Omega$ and 
$\ol{u}=u+\sigma$, for all $\sigma>0,$
 then there exists a structural constant $c>0$ such that
\begin{equation*}
\sup_{B_R(x_0) }\ol{u}\leq c\biggl(\fint_{B_{2R}(x_0)} \ol{u}^2\biggr)^{\frac{1}{2}}
\end{equation*}
for every $\ol{B_{4R}(x_0)}\subset\Omega$, $R<r_0$. 

\end{lem} 

Proof of this result is given in \cite[p.175]{U}.

We recall that by Definition \ref{def structural constant} a structural constant does not depend on $u\in \K_{\Omega, g+\sum_{i=1}^N\partial_i f_i}$ nor on the balls defined by the quasi distance. On the other hand it may depend on the ellipticity constants $\lambda,\; \Lambda$, the doubling constant $C_D$, the constant in the Poincaré inequality, the Lipschitz constant of vector fields $\{X_i\}_{1,\dots, m}$  and  $(\|\gamma\|^2_{2p}+1)^{1/2}$.

\begin{lem}[Fabes Lemma, \cite{DGL} Lemma 7.4]\label{Fabes Lemma} Let $v\in W^{1}_{loc}(\Omega,X)$, $B_R(x_0)\subset\Omega$ and assume there exists $0<\e\leq 1$ such that
\begin{equation*}
|\{x\in B_{R}(x_0):v(x)=0\}|\geq \e|B_R|
\end{equation*}
then
there exists a constant $C\geq 0$ depending on $\e$ such that
\begin{equation*}
\fint_{B_R(x_0)} |v|^2\;dx\leq CR^2\fint_{B_R(x_0)}|Xv|^2dx.
\end{equation*}
\end{lem}

\begin{lem}[Estimates for $ \|X \log \ol{u}\|_{L^2(B_R(x_0))}$]\label{Estimates X logu Lemma}
Let $u$ be a non negative weak subsolution to \eqref{non omog X elliptic} in $\Omega$, then there exists a structural constant $c>0$ such that
\begin{equation*}
\fint_{B_R(x_0)}|X\log\ol{u}|^2\;dx \leq \frac{c}{R^2}
\end{equation*}
for every $\ol{B_{4R}}(x_0)\subset\Omega$, $R<r_0$. Here $\ol{u}=u+\mathcal{S}_{\Omega}\Bigl(B_{R}(x_0),g+\sum_{i=1}^N\partial_i f_i\Bigr)$.
\end{lem}
\begin{proof}
In \cite[p.176]{U} it is shown that
\begin{equation*}
\int_{B_{4R}(x_0)}|\eta X\log\ol{u}|^2\; dx\leq C\int_{B_{4R}(x_0)} F\; dx
\end{equation*} 
where $\eta\in C^1_0(B_{4R}(x_0))$ is a non negative function, $F=(1+6\Lambda^2/\lambda^2)|X\eta|^2+\eta^2a^2$, $a=\bigl(\frac{1}{\lambda^2}\bigl(\gamma^2+\bigl(\frac{\gamma_0}{\sigma}\bigr)^2\bigr)+\frac{|g|}{\lambda\sigma}\bigr)^{\frac{1}{2}}$ and $C$ may depend on $a^*(R)=\sup_{\rho\leq 4R}\frac{\rho}{|B_{\rho}|^{1/2p}}\|a\|_{L^{2p}(B_{\rho})}$. Starting from this estimate it suffices to choose $0\leq\eta\leq 1$, $\eta\equiv 1$ in $B_R(x_0)$,  $\eta\equiv 0$ in $\Omega\setminus B_{2R}(x_0)$, $|X\eta|\leq \frac{C}{R}$ and recall the doubling condition to obtain
\begin{align*}
\fint_{B_R(x_0)}|X\log\ol{u}|^2\;dx&\leq CC_D^2\fint_{B_{2R}(x_0)} F\;dx\\
&\leq C \fint_{B_{2R}(x_0)}\frac{1}{R^2}+a^2\;dx\\
&=\frac{C}{R^2}\biggl(R^2\fint_{B_{2R}(x_0)}a^2\;dx+1\biggr).
\end{align*}
Then, by using H\"older inequality, we estimate the right hand side of inequality above obtaining
\begin{align*}
\fint_{B_R(x_0)}|X\log\ol{u}|^2\;dx&\leq \frac{C}{R^2}\biggl(R^2\biggl(\fint_{B_{2R}(x_0)}a^{2p}\;dx\biggr)^{1/p}+1\biggr)\\
&\leq \frac{C}{R^2}\bigl((a^*(R))^2+1\bigr).
\end{align*}
We remark that in \cite[p.177]{U}  it is shown that $a^*(R)\leq C(\|\gamma\|^2_{2p}+1)^{1/2}$ so it is bounded from above by a structural constant.
\end{proof}

\begin{thm}\label{thm crit density X elliptic} Let $\eta>4$, $R<r_0$ and $u$ be a non negative supersolution to \eqref{non omog X elliptic} in $B_{\eta R}(x_0)\subset\Omega$ and  $0<\nu<1$, if $u$ satisfies 
\begin{equation*}
|\{x\in B_R(x_0):u(x)\geq 1\}|\geq\nu|B_R(x_0)|
\end{equation*}
then there exist two constant $0<c,\e<1$ depending on $\nu$ such that
%\footnote{$\e_0$ dipenda da $e^{-c}$ e quindi dalla scelta di $\nu$ in quanto per il Fabes Lemma $c=c(\nu)$. Questo si risolve fissando $\nu<1/C_D^2$, e di conseguenza $\e_0$}
\begin{equation*}
\inf_{B_{R/2}}u \geq c\quad\text{or}\quad \sigma(R)\geq\e.
\end{equation*}
Here $\sigma(R)=:\mathcal{S}_{\Omega}\Bigl(B_{x_0}(R),\;g+\sum_{i=1}^N\partial_i f_i\Bigr)$.
\end{thm}

\begin{proof}
Define $\ol{u}=u+\sigma(R)$ and $h(\ol{u}):=\max\{-\log \ol{u}, 0\}$. It is easy to see that $w:=h(\ol{u})$ satisfies $\mathfrak{L}(w,v)\leq -\mathfrak{F}(v)$. On one hand  
\begin{equation*}
|\{x\in B_r(x_0): w(x)=0\}|=|\{x\in B_R(x_0): \ol{u}(x)\geq 1\}|\geq |\{x\in B_R(x_0): u(x)\geq 1\}|\geq \nu |B_R(x_0)|,
\end{equation*}
so that, by Fabes Lemma we have
\begin{equation*}
\fint_{B_R(x_0)}|w(x)|^2\;dx\leq CR^2\fint_{B_R(x_0)} |Xw(x)|^2\;dx\leq CR^2\fint_{B_R(x_0)} |X\log(\ol{u}(x))|^2\;dx.
\end{equation*}
On the other hand Lemma \ref{Local boundness Lemma} implies
\begin{equation*}
\sup_{B_{R/2}(x_0)}w\leq c\biggl(\fint_{B_R(x_0)}h^2(u)\;dx\biggr)^{\frac{1}{2}}.
\end{equation*}
Concatenating inequalities above and recalling that from Lemma \ref{Estimates X logu Lemma} follows
\begin{equation*}
\fint_{B_R(x_0)}|X\log\ol{u}|^2\;dx\leq \frac{C}{R^2}
\end{equation*}
we find
\begin{equation*}
\sup_{B_{R/2}(x_0)}w\leq cR\biggl(\fint_{B_R(x_0)}|X\log\ol{u}|^2\;dx \biggr)^{\frac{1}{2}}\leq c
\end{equation*}
from which
\begin{equation*}
\inf_{B_{R/2}(x_0)}\ol{u}(x)\geq e^{-c}=:c_0.
\end{equation*}
So that, if $\sigma(R)< \e\leq \frac{c_0}{2}$ we find $\inf_{B_{R/2}(x_0)}u(x)\geq c_0/2$.
\end{proof}

It is easy to see that for every non negative $u\in \K_{\Omega, g+\sum_{i=1}^N\partial_i f_i} $ if $\lambda\geq 0$ we have $\lambda u\in \K_{\Omega,\lambda( g+\sum_{i=1}^N\partial_i f_i)} $ and if $\tau-\lambda u\geq 0$ then $\tau-\lambda u\in \K_{\Omega,-\lambda(g+\sum_{i=1}^N\partial_i f_i)} $. Moreover, structural constants in Theorem \ref{thm crit density X elliptic} are independent of $g$ and of $\partial_if_i$. % Theorem  if we consider $\mathcal{S}_{\Omega}(B_{x_0}(r),g+\sum_{i=1}^N\partial_i f_i):=r^{\delta}\|\gamma_0\|_{L^{2p}(\Omega)}+r^{2\delta}\|g\|_{L^p(\Omega)}$, families $\K_{\Omega,\lambda(g+\sum_{i=1}^N\partial_i f_i)}$ and $\K_{\Omega,-\lambda(g+\sum_{i=1}^N\partial_i f_i)}$ have the $\nu$ critical density property $CD(\nu,c,\e,\eta)$ for every fixed $0<\nu<1$ and for every $\lambda\geq 0$. Indeed it suffices to apply Theorem \ref{thm crit density X elliptic} to functions $u_{\lambda}:=\lambda u$ and $u_{-\lambda}:=\tau-\lambda u$ respectively.\\

Since  $d$ is the Carnot–Carathéodory distance $d$ related to the family of vectors fields $X$ and we are assuming that it is continuous with respect to the Euclidean topology, by 
 \cite[Lemma 3.7]{GN}  and \cite[Lemma 2.8]{DGL} 
 the function $r\to \mu(B_r(x))$ is continuous and we can apply Theorems \ref{teo: power decay} 
and \ref{Harnack} to the family of functions $\K_{\Omega, g+\sum_{i=1}^N\partial_i f_i}$ to get the following Harnack inequality.

\begin{thm}
Let $\Omega\subset\mathbb{R}^N$ be a domain in which the reverse doubling property holds and $u \in W^1_{\text{loc}}(\Omega)$ be a non negative weak solution to $Lu=g+\sum_{i=1}^N\partial_i f_i$ in $\Omega$, $r<r_0=r_0(K_0)$. Then there exists a structural constant $\eta>4$ such that for every $B_{\eta r}\subseteq\Omega$ we have
\begin{equation*}
\sup_{B_r}u\leq C(\inf_{B_r}u+r^{\delta}\|\gamma_0\|_{L^{2p}(\Omega)}+r^{2\delta}\|g\|_{L^p(\Omega)}).
\end{equation*}
Here $K_0\subset\mathbb{R}^N$ is a fixed compact set whose interior contains the closure of $\Omega$, $r_0$ is chosen small enough so the Sobolev Inequality holds for every $B_{4r}(x)\subset\Omega$, $\delta=1-\frac{Q}{2p}$ and $C$ is a structural constant.
\end{thm}


\begin{thebibliography}{99}
\addcontentsline{toc}{chapter}{Bibliography}

\bibitem{AGT} F. Abedin, C.E. Gutiérrez, G. Tralli, \emph{Harnack's inequality for a class of non-divergent equations in the Heisenberg group}, arXiv:1705.10856

\bibitem{AFT} H. Aimar, L. Forzani, R. Toledano, \emph{Holder regularity of solutions of PDE’s: a geometrical view}. Commun.
Partial Differ. Equ. 26 (2002), 1145–1173.

\bibitem{BCK15} Z.M. Balogh, A. Calogero, A. Kristaly, \emph{Sharp comparison and maximum principles via horizontal
normal mapping in the Heisenberg group}, J. Funct. Anal., Vol. 269, Issue 9, (2015) 2669-2708.

\bibitem{BB} E. Battaglia, A. Bonfiglioli, \emph{An invariant Harnack inequality for a class of subelliptic operators under global doubling and Poincaré assumptions, and applications}, preprint 

\bibitem{C97} X. Cabré, \emph{Nondivergent elliptic equations on manifolds with nonnegative curvature}, Comm. Pure
Appl. Math., Vol. 50, Issue 7, (1997) 623-665.

\bibitem{C89} L.A. Caffarelli, \emph{Interior a priori estimates for solutions of fully nonlinear elliptic equations}, Ann. Math. 130 (1989), 189–213 

\bibitem{CC95} L.A. Caffarelli,  X. Cabré, \emph{Fully nonlinear elliptic equations}, volume 43 of American Mathematical Society Colloquium Publications. American Mathematical Society, Providence, RI, (1995), vi-104.

\bibitem{CG97} L.A. Caffarelli, C.E. Gutiérrez, \emph{Properties of the solutions of the linearized Monge–Ampère equation}, Am. J. Math. 119 (2) (1997), 423–465. 

\bibitem{CW} R. Coifman, G. Weiss, \emph{Analyse harmonique non-commutative sur certains espaces homogenes},
  volume 242 of Lecture Notes in Mathematics, Berlin-New York,Springer-Verlag (1971).

\bibitem{DGN03} D. Danielli, N. Garofalo, D.-M. Nhieu, \emph{On the best possible character of the $L^Q$ norm in some
a priori estimates for non-divergence form equations in Carnot groups,} Proc. Amer. Math. Soc., Vol.
131, Issue 11, (2003), 3487-3498.

\bibitem{DGL} G. Di Fazio, C.E. Gutiérrez, E. Lanconelli, \emph{Covering theorems, inequalities on metric spaces and application to PDE's}, Math. Ann. 341 (2008), 255–291.

\bibitem{FL} B. Franchi, E. Lanconelli, \emph{Une métrique associée à une classe d'opérateurs elliptiques dégénerés}, (Conference on Linear Partial and Pseudo Differential Operators (Torino, 1982)), Rend. Sem. Mat. Univ. e Politec. Torino, (1983) 105-114, (Special Issue). 

\bibitem{FLW} B. Franchi, G. Lu,  R.L. Wheeden, \emph{A relationship between Poincaré type inequalities and representation formulas in metric spaces}, Int. Math. Res. Not. (1996), 1-14.

\bibitem{GN} N. Garofalo, D.M Nhieu, \emph{Isoperimetric and Sobolev inequalities for Carnot Carathéodory spaces and the existence of minimal surfaces}, Comm. Pure Appl. Math., 49 (1996), 1081-114. 


\bibitem{GT01}  D. Gilbarg, N.S. Trudinger, \emph{Elliptic partial differential equations of second order}, Classics
in Mathematics, Springer-Verlag, Berlin, 2001 (Reprint of the 1998 edition).

\bibitem{G16} C.E. Gutiérrez,
\emph{The Monge-Ampère equation}, Progress in Nonlinear Differential Equations
and their Applications, Birkh\"auser, Boston, 2nd edition, 2016.



\bibitem{GL}
C.E. Gutiérrez, E. Lanconelli,\emph{ Maximum principle, nonhomogeneous Harnack inequality, and Liouville theorems for X-elliptic operators}. Comm. Partial Differ. Equ. 28 (2003), 1833–1862.

 \bibitem{GM04} C.E. Gutiérrez, A. Montanari,  \emph{Maximum and comparison principles for convex functions on the
Heisenberg group}, Comm. Partial Differential Equations, Vol. 29, Issue 9-10, (2004) 1305-1334.



\bibitem{GT} C.E. Gutiérrez, F. Tournier, \emph{Harnack inequality for a degenerate elliptic equation}, Comm. Partial Differential Equations 36 no. 12 (2011), 2103–2116.

\bibitem{HK} P. Hajlasz, P. Koskela, \emph{Sobolev met Poincaré} Mem. Amer. Math. Soc. 145(688):x +101 (2000)

\bibitem{H} A. Harnack, \emph{Die Grundlagen der Theorie des logarithmischen Potentiales und der eindeutigen Potentialfunktion in der Ebene}, (1887).

%\bibitem{GT} D. Gilbarg, N. S. Trudinger, \emph{Elliptic partial differential equations of second order}. New York, Springer-Verlag, Second Edition, (1983).

%\bibitem{CG} C.E Gutiérrez, \emph{The Monge—Ampère Equation}. Vol. 44. Springer Science \& Business Media, (2012).

\bibitem{IMS} S. Indratno, D. Maldonado, S. Silwal, \emph{On the axiomatic approach to Harnack's inequality in doubling quasi-metric spaces}, J. Differential Equations 254 (2013), no. 8, 3369–3394. 

\bibitem{KS79} N. Krylov and M. Safonov, \emph{An estimate on the probability that a diffusion hits a set of positive measure}, Soviet Math. 20 (1979), 253-256.

\bibitem{KS80} N. Krylov and M. Safonov, \emph{A property of the solutions of parabolic equations with measurable coefficients}, Izv. Akad. Nauk SSR Ser. Mat. 44:1 (1980), 161-175.
 
\bibitem{MS} R. A. Macías, C. Segovia, \emph{Lipschitz functions on spaces of homogeneous type}. Adv. in Math. 33 (1979), no. 3, 257–270

\bibitem{AM} A. Montanari, \emph{Harnack inequality for a subelliptic PDE in nondivergence form}, Nonlinear Anal. 109 (2014), 285–300.

\bibitem{JM} J. Moser, \emph{On a pointwise estimate for parabolic differential equations}, Comm. Pure Appl. Math. 24 (1971), 727–740.

\bibitem{T} G. Tralli, \emph{A certain critical density property for invariant Harnack inequalities in H-type groups}.
J. Differential Equations 256 no. 2 (2014), 461–474. 

\bibitem{TZ} N. S. Trudinger,  W. Zhang, \emph{Hessian measures on the Heisenberg group}. J. Funct. Anal. 264 (2013), no. 10, 2335-2355. 

\bibitem{U} F. Uguzzoni, \emph{Estimates of the Green function for X-elliptic operators}, Math. Ann. 361 (2015), 169-190.

%\bibitem{WZ} Y. Wang, X. Zhang, \emph{An Alexandroff-Bakelman-Pucci estimate on Riemannian manifolds}, Adv. Math. 232 (2013), 499–512.


\end{thebibliography}
\end{document}